\documentclass[11pt,reqno,final]{amsart}

\usepackage{amsmath, amssymb, epsfig}
\usepackage{mathrsfs}
\usepackage{color}
\usepackage{alg}
\usepackage{url}
\usepackage[font=small,margin=10pt,labelfont={bf},labelsep={space}]{caption}
\usepackage{subfig}
\usepackage{wrapfig}
\usepackage[scaled]{helvet} 
\usepackage{fourier} 
\usepackage{bm} 

\makeatletter
\def\algspacing{\alg@unmargin}
\makeatother

\usepackage[margin=1in]{geometry}
\newlength{\algorithmwidth}
\theoremstyle{plain}
\newtheorem{theorem}{Theorem}[section]
\newtheorem{proposition}[theorem]{Proposition}
\newtheorem{corollary}[theorem]{Corollary}

\theoremstyle{definition}

\theoremstyle{remark}
\newtheorem*{remark}{Remark}
\numberwithin{theorem}{section}
\numberwithin{equation}{section}

\DeclareMathOperator*{\argmin}{arg min}

\newcommand{\defby}{\overset{\mathrm{\scriptscriptstyle{def}}}{=}}

\newcommand{\R}{\mathbb{R}}

\newcommand{\cD}{\mathcal{D}}

\newcommand{\bB}{{\boldsymbol{B}}}

\newcommand{\bD}{\boldsymbol{D}}

\newcommand{\bX}{\boldsymbol{X}}

\newcommand{\bx}{\boldsymbol{x}}
\newcommand{\bPsi}{\boldsymbol{\Psi}}

\newcommand{\bA}{\boldsymbol{A}}
\newcommand{\by}{\boldsymbol{y}}

\newcommand{\be}{\boldsymbol{e}}

\newcommand{\ba}{\boldsymbol{a}}
\newcommand{\bb}{\boldsymbol{b}}

\newcommand{\bz}{\boldsymbol{z}}
\newcommand{\bw}{\boldsymbol{w}}

\newcommand{\xls}{\boldsymbol{x}_{LS}}
\newcommand{\xhl}{\boldsymbol{x}_{HL}}

\newcommand{\vct}[1]{\bm{#1}}
\newcommand{\mtx}[1]{\bm{#1}}

\def \xcur {\vct{x}_k}
\def \xnext {\vct{x}_{k+1}}

\def \A {\mtx{A}}

\begin{document}

\title[]{Batched Stochastic Gradient Descent with Weighted Sampling} %Feel free to change
%Edited: 3/26/13 D. Needell
\author{Deanna Needell and Rachel Ward}
\date{\today}

\begin{abstract}
We analyze a batched variant of Stochastic Gradient Descent (SGD) with weighted sampling distribution for smooth and non-smooth objective functions.  We show that by distributing the batches computationally, a significant speedup in the convergence rate is provably possible compared to either batched sampling or weighted sampling alone.  We propose several computationally efficient schemes to approximate the optimal weights, and compute proposed sampling distributions explicitly for the least squares and hinge loss problems.  We show both analytically and experimentally that substantial gains can be obtained.
\end{abstract}

\maketitle

\section{Mathematical Formulation}

We consider minimizing an objective function of the form
\begin{equation}\label{obj}
F(\bx) = \frac{1}{n}\sum_{i=1}^n f_i(\bx) = \mathbb{E} f_i(\bx).
\end{equation}
One important such objective function is the least squares objective for linear systems.  Given an $n\times m$ matrix $\bA$ with rows $\ba_1, \ldots, \ba_n$ and a vector $\bb\in\R^n$, one searches for the least squares solution $\xls$ given by
\begin{equation}\label{LS}
\xls \defby \argmin_{\bx\in\R^m} \frac{1}{2}\|\bA\bx - \bb\|_2^2 = \argmin_{\bx\in\R^m}\frac{1}{n}\sum_{i=1}^n \frac{n}{2}(b_i - \langle\ba_i, \bx\rangle)^2 = \argmin_{\bx\in\R^m}\mathbb{E} f_i(\bx),
\end{equation}
where the functionals are defined by $f_i(\bx) = \frac{n}{2}(b_i - \langle\ba_i, \bx\rangle)^2$.

Another important example is the setting of support vector machines where one wishes to minimize the hinge loss objective given by
\begin{equation}\label{HL}
\xhl \defby \argmin_{\bw\in\R^m} \frac{1}{n}\sum_{i=1}^n[y_i\langle \bw, \bx_i\rangle]_+ + \frac{\lambda}{2}\|\bw\|_2^2.
\end{equation}
Here, the data is given by the matrix $\bX$ with rows $\bx_1, \ldots, \bx_n$ and the labels $y_i \in \{-1, 1\}$.  The function $[z]_+ \defby \max(0, z)$ denotes the positive part.  We view the problem \eqref{HL} in the form \eqref{obj} with $f_i(\bw) = [1- y_i\langle \bw, \bx_i\rangle]_+$ and regularizer $\frac{\lambda}{2}\|\bw\|_2^2$.

The stochastic gradient descent (SGD) method solves problems of the form \eqref{obj} by iteratively moving in the gradient direction of a randomly selected functional.  SGD can be described succinctly by the update rule:
$$
\xnext \leftarrow \xcur - \gamma\nabla f_{i_k}(\xcur),
$$
where index $i_k$ is selected randomly in the $k$th iteration, and an initial estimation $\vct{x}_0$ is chosen arbitrarily.  Typical implementations of SGD select the functionals uniformly at random, although if the problem at hand allows a one-pass preprocessing of the functionals, certain \textit{weighted} sampling distributions preferring functionals with larger variation can provide better convergence (see e.g. \cite{needell2014stochastic,zhao2014stochastic} and references therein).  In particular, Needell et al. show that selecting a functional with probability proportional to the Lipschitz constant of its gradient yields a convergence rate depending on the \textit{average} of all such Lipschitz constants, rather than the supremum \cite{needell2014stochastic}.  An analogous result in the same work shows that for non-smooth functionals, the probabilities should be chosen proportional to the Lipschitz constant of the functional itself.

Another variant of SGD utilizes so-called \textit{mini-batches}; in this variant, a batch of functionals is selected in each iteration rather than a single one \cite{cotter2011better,agarwal2011distributed,dekel2012optimal,takavc2013mini}.  The computations over the batches can then be run in parallel and speedups in the convergence are often quite significant. 

{\bfseries Contribution.}  We propose a weighted sampling scheme to be used with mini-batch SGD.  We show that when the batches can be implemented in parallel, significant speedup in convergence is possible.  In particular, we analyze the convergence using \textit{efficiently computed distributions} for the least squares and hinge loss objectives, the latter being especially challenging since it is non-smooth.  We demonstrate theoretically and empirically that weighting the distribution and utilizing batches of functionals per iteration together form a complementary approach to accelerating convergence, highlighting the precise improvements and weighting schemes for these settings of practical interest.

{\bfseries Organization.} We next briefly discuss some related work on SGD, weighted distributions, and batching methods.  We then combine these ideas into one cohesive framework and discuss the benefits in various settings.  Section \ref{sec:weights} focuses on the impact of weighting the distribution.  In Section \ref{sec:smooth} we analyze SGD with weighting and batches for smooth objective functions, considering the least squares objective as a motivating example.  We analyze the non-smooth case along with the hinge loss objective function in Section \ref{sec:non-smooth}.  We display experimental results for the least squares problem in Section \ref{sec:exps} that serve to highlight the relative tradeoffs of using both batches and weighting, along with different computational approaches.  We conclude in Section \ref{sec:conc}.

{\bfseries Related work.}  
Stochastic gradient descent, stemming from the work \cite{robmon}, has recently received renewed attention for its effectiveness in treating large-scale problems arising in machine learning \cite{bottou, bottou2010large, njls09, shalev2008svm}.  Importance sampling in stochastic gradient descent, as in the case of mini-batching (which we also refer to simply as \textit{batching} here), also leads to variance reduction in stochastic gradient methods and, in terms of theory, leads to improvement of the leading constant in the complexity estimate, typically via replacing  the maximum of certain data-dependent quantities by their average.  Such theoretical guarantees were shown for the case of solving least squares problems where stochastic gradient descent coincides with the randomized Kaczmarz method in \cite{SV09:Randomized-Kaczmarz}.  This method was extended to handle noisy linear systems in \cite{Nee10:Randomized-Kaczmarz}.  Later, this strategy  was extended to the more general setting of  smooth and strongly convex objectives in \cite{needell2014stochastic}, building on an analysis of stochastic gradient descent in \cite{bach2011}.   Later, \cite{zhao2014stochastic} considered a similar importance sampling strategy for convex but not necessarily smooth objective functions.  Importance sampling has also been considered in the related setting of stochastic coordinate descent/ascent methods \cite{nesterov2012efficiency, richtak14, qrz15, cqr15}.  Other papers exploring advantages of importance sampling in various adaptations of stochastic gradient descent include but are not limited to \cite{lee2013efficient, schmidt2013minimizing, xiao2014proximal, defossez2015averaged}.

Mini-batching in stochastic gradient methods refers to pooling together several random examples in the estimate of the gradient, as opposed to just a single random example at a time, effectively reducing the variance of each iteration \cite{shalev2011pegasos}.   On the other hand, each iteration also increases in complexity as the size of the batch grows.  However, if parallel processing is available, the computation can be done concurrently at each step, so that the ``per-iteration cost" with batching is not higher than without batching.  Ideally, one would like the consequence of using batch size $b$ to result in a convergence rate speedup by factor of $b$, but this is not always the case \cite{byrd2012sample}.  Still, \cite{takavc2013mini} showed that by incorporating parallelization or multiple cores, this strategy can only improve on the convergence rate over standard stochastic gradient, and can improve the convergence rate by a factor of the batch size in certain situations, such as when the matrix has nearly orthonormal rows.  Other recent papers exploring the advantages of mini-batching in different settings of stochastic optimization include \cite{cotter2011better, dekel2012optimal, NW12:Two-Subspace-Projection, konevcny2014ms2gd, li2014efficient}. 

The recent paper \cite{csiba2016importance} also considered the combination of importance sampling and mini-batching for a stochastic dual coordinate ascent algorithm in the general setting of empirical risk minimization, wherein the function to minimize is smooth and convex.  There the authors provide a theoretical optimal sampling strategy that is not practical to implement but can be approximated via alternating minimization.  They also provide a computationally efficient formula that yields better sample complexity than uniform mini-batching, but without quantitative bounds on the gain.   In particular, they do not provide general assumptions under which one achieves provable speedup in convergence depending on an average Lipschitz constant rather than a maximum. 

For an overview of applications of stochastic gradient descent and its weighted/batched variants in large-scale matrix inversion problems, we refer the reader to \cite{gower2016randomized}. 

\section{SGD with weighting}\label{sec:weights}

Recall the objective function \eqref{obj}.  We assume in this section that the function $F$ and the functionals $f_i$ satisfy the following convexity and smoothness conditions:
%\begin{table}[ht]

\begin{center}
\emph{Convexity and smoothness conditions}
\end{center}
%\caption*{Convexity and smoothness conditions}
%\label{table:assume}
%\begin{tabular}{p{16cm}}
%\vspace{-0.2in}
\begin{enumerate}
\item Each $f_i$ is continuously differentiable and the gradient function $\nabla f_i$ has Lipschitz 
constant bounded by $L_i$: $\| \nabla f_i(\bx) - \nabla f_i(\by) \|_2 \leq L_i \| \bx - \by \|_2$ for all vectors $\bx$ and $\by$.
\item $F$ has strong convexity parameter $\mu$; that is, $\langle \bx - \by, \nabla F(\bx) - \nabla F(\by) \rangle \geq \mu \| \bx - \by \|_2^2$ for all vectors $\bx$ and $\by$.
\item At the unique minimizer $\bx_{*} = \argmin F(\bx)$, the average gradient norm squared $\| \nabla f_i(\bx_{*}) \|_2^2$ is not too large, in the sense that 
$$
\frac{1}{n} \sum_{i=1}^n \| \nabla f_i(\bx_{*}) \|_2^2 \leq \sigma^2.
$$
\end{enumerate}
%\end{tabular}
%\end{table}

An unbiased gradient estimate for $F(\bx)$ can be obtained by drawing $i$ uniformly from $[n] \defby \{1, 2, \ldots, n\}$ and using $\nabla f_i(\bx)$ as the estimate for $\nabla F(\bx)$.  The standard SGD update with fixed step size $\gamma$ is given by
\begin{equation}
\label{sgd_update}
\bx_{k+1} \leftarrow \bx_{k} - \gamma \nabla f_{i_k}(\bx_k)
\end{equation}
where each $i_k$ is drawn uniformly from $[n]$.  The idea behind weighted sampling is that, by drawing $i$ from a weighted distribution 
$\cD^{(p)} = \{p(1), p(2), \dots, p(n)\}$ over $[n]$, the weighted sample $\frac{1}{p(i_k)} \nabla f_{i_k}(\bx_k)$ is still an unbiased estimate of the gradient $\nabla F(\bx)$.  This motivates the weighted SGD update 
\begin{equation}
\label{sgd_update_w}
\bx_{k+1} \leftarrow \bx_{k} - \frac{\gamma}{n p(i_k)} \nabla f_{i_k}(\bx_k),
\end{equation}
In \cite{needell2014stochastic}, a family of distributions $\cD^{(p)}$ whereby functions $f_i$ with larger Lipschitz constants are more likely to be sampled was shown to lead to an improved convergence rate in SGD over uniform sampling.  In terms of the  distance $\| \bx_k - \bx_* \|_2^2$ of the $k$th iterate to the unique minimum, starting from initial distance $\varepsilon_0 = \| \bx_0 - \bx_{*} \|_2^2$, Corollary 3.1 in \cite{needell2014stochastic} is as follows.

\begin{proposition}
\label{prop:main}
Assume the convexity and smoothness conditions are in force.  For any desired $\varepsilon > 0$, and using a stepsize of 
$$
\gamma = \frac{\mu \varepsilon}{4 (\varepsilon \mu \frac{1}{n} \sum_{i=1}^n L_i + \sigma^2)},
$$
we have that after
\begin{equation}
\label{k_num}
k = \left\lceil4 \log(2 \varepsilon_0/\varepsilon) \left( \frac{ \frac{1}{n} \sum_{i=1}^n L_i }{\mu} + \frac{\sigma^2}{\mu^2 \varepsilon}  \right)\right\rceil
\end{equation}
iterations of weighted SGD \eqref{sgd_update_w} with weights 
\begin{equation}
\label{w:weights}
p(i) = \frac{1}{2n} + \frac{1}{2n} \cdot \frac{L_i}{\frac{1}{n} \sum_i L_i},
\end{equation}
the following holds in expectation with respect to the weighted distribution \eqref{w:weights}: $\mathbb{E}^{(p)} \| {\bf x}_k - {\bf x}_{*} \|_2^2 \leq \varepsilon$.
\end{proposition}

\begin{remark}
Note that one can obtain a gaurantee with the same number of iterations as \eqref{k_num}, same weights \eqref{w:weights}, and step-sizes which depend on the index selected by cleverly re-writing the objective function as the sum of scaled functionals $f_i$ each repeated an appropriate number of times. We state the version in Proposition \ref{prop:main} derived from \cite{needell2014stochastic} here for simplicity and convenience, and note that it improves upon classical results even in the uniformly bounded Lipschitz case.
\end{remark}

\begin{remark}
This should be compared to the result for uniform sampling SGD \cite{needell2014stochastic}: using step-size 
$\gamma =  \frac{\mu \varepsilon}{4 (\varepsilon \mu (\sup_i L_i)+ \sigma^2)}$,
one obtains the comparable error guarantee $\mathbb{E} \| {\bf x}_k - {\bf x}_{*} \|_2^2 \leq \varepsilon$  after a number of iterations
\begin{equation}
\label{k_uniform}
k = \left\lceil2 \log(2 \varepsilon_0/\varepsilon) \left( \frac{ \sup_i L_i }{\mu} + \frac{\sigma^2}{\mu^2 \varepsilon}  \right)\right\rceil.
\end{equation}
Since the average Lipschitz constant ${\frac{1}{n} \sum_i L_i}$ is always at most $\sup_i L_i$, and can be up to $n$ times smaller than $\sup_i L_i$, SGD with weighted sampling requires twice the number of iterations of uniform SGD in the worst case, but can potentially converge much faster, specifically, in the regime where
$$
\frac{\sigma^2}{\mu^2 \varepsilon} \leq \frac{ \frac{1}{n} \sum_{i=1}^n L_i }{\mu} \ll  \frac{ \sup_i L_i }{\mu}. 
$$
\end{remark}

\section{Mini-batch SGD with weighting: the smooth case}\label{sec:smooth}

Here we present a weighting and mini-batch scheme for SGD based on Proposition \ref{prop:main}.  For practical purposes, we assume that the functions $f_i(\bx)$ such that $F(\bx) = \frac{1}{n}\sum_{i=1}^n f_i(\bx)$ are initially partitioned into fixed batches of size $b$ and denote the partition by $\{\tau_1, \tau_2, \ldots \tau_d\}$ where $|\tau_i| = b$ for all $i < d$ and $d = \lceil n/b\rceil$ (for simplicity we will henceforth assume that $d = n/b$ is an integer).  We will randomly select from this pre-determined partition of batches; however, our analysis extends easily to the case where a batch of size $b$ is randomly selected each time from the entire set of functionals.  With this notation, we may re-formulate the objective given in \eqref{obj} as follows:
$$
F(\bx) = \frac{1}{d}\sum_{i=1}^d g_{\tau_i}(\bx) = \mathbb{E} g_{\tau_i}(\bx),
$$
where now we write $g_{\tau_i}(\bx) = \frac{1}{b}\sum_{j\in\tau_i} f_j(\bx)$.  We can apply Proposition \ref{prop:main} to the functionals $g_{\tau_i}$, and select batch $\tau_i$ with probability proportional the Lipschitz constant of $\nabla g_{\tau_i}$ (or of $g_{\tau_i}$ in the non-smooth case, see Section \ref{sec:non-smooth}).   Note that 
\begin{itemize}
\item The strong convexity parameter $\mu$ for the function $F$ remains invariant to the batching rule.
\item  The residual error $\sigma^2_{\tau}$ such that 
$ \frac{1}{d} \sum_{i=1}^d \| \nabla g_{\tau_i} (\bx_{*}) \|_2^2 \leq \sigma^2_{\tau}$ can only {\bf decrease} with increasing batch size, since for $b\geq 2$,
\begin{equation}
\label{sigma:down}
\sigma^2_{\tau} = \frac{1}{d} \sum_{i=1}^d \left\| \frac{1}{b} \nabla \left( \sum_{k \in\tau_i} f_k(\bx_{*}) \right) \right\|_2^2 \leq  \frac{1}{n}\sum_{i=1}^n  \| \nabla f_i(\bx_{*}) \|_2^2 \leq \sigma^2.  \nonumber \\
\end{equation}
Note that for certain objective functions, this bound can be refined with a dependence on the block size $b$, in which case even further improvements can be gained by batching, see e.g. \eqref{ortho} and surrounding discussions. 
\item The average Lipschitz constant $\overline{L}_{\tau} = \frac{1}{d} \sum_{i=1}^d L_{\tau_i}$ of the gradients of the batched functions $g_{\tau_i}$ can only {\bf decrease} with increasing batch size, since by the triangle inequality, $L_{\tau_i} \leq \frac{1}{b} \sum_{k \in \tau_i} L_k,$ and thus
$$
\frac{1}{d} \sum_{i=1}^d L_{\tau_i}  \leq \frac{1}{n} \sum_{k=1}^n L_k = \overline{L}.
$$
\end{itemize}
Incorporating these observations, applying Proposition \ref{prop:main} in the batched weighted setting implies that incorporating weighted sampling and mini-batching in SGD results in a convergence rate that equals or improves on the rate obtained using weights alone:

\begin{theorem}
\label{thm:batchweight}
Assume that the convexity and smoothness conditions on $F(\bx) = \frac{1}{n}\sum_{i=1}^n f_i(\bx)$ are in force.  Consider the $d = n/b$ batches $g_{\tau_i}(\bx) = \frac{1}{b}\sum_{k \in\tau_i} f_k(\bx)$, and the batched weighted SGD iteration 
$$
\bx_{k+1} \leftarrow \bx_{k} - \frac{\gamma}{d \cdot p(\tau_{i_k})} \nabla g_{\tau_{i_k}}(\bx_k)
$$
where batch $\tau_i$ is selected at iteration $k$ with probability
\begin{equation}\label{newweights}
%\label{w:batchweights}
p(\tau_i) = \frac{1}{2d} + \frac{1}{2d} \cdot \frac{L_{\tau_i}}{\overline{L}_{\tau}}.
\end{equation}  
For any desired $\varepsilon$, and using a stepsize of 
$$
\gamma = \frac{\mu \varepsilon}{4 (\varepsilon \mu \overline{L}_{\tau}+ \sigma_{\tau}^2)},
$$
we have that after a number of iterations
\begin{equation}\label{kwoot}
k = \left\lceil4 \log(2 \varepsilon_0/\varepsilon) \left( \frac{ \overline{L}_{\tau} }{\mu} + \frac{\sigma_{\tau}^2}{\mu^2 \varepsilon}  \right)\right\rceil \leq \left\lceil4 \log(2 \varepsilon_0/\varepsilon) \left( \frac{ \overline{L} }{\mu} + \frac{\sigma^2}{\mu^2 \varepsilon}  \right)\right\rceil,
\end{equation}
the following holds in expectation with respect to the weighted distribution \eqref{newweights}: $\mathbb{E}^{(p)} \| {\bf x}_k - {\bf x}_{*} \|_2^2 \leq \varepsilon$.
Since 
\end{theorem}

\begin{remark}
The inequality in \eqref{kwoot} implies that batching and weighting can only improve the convergence rate of SGD compared to weighting alone. As a reminder, this is under the assumption that the batches can be computed in parallel, so depending on the number of cores available, one needs to weigh the computational tradeoff between iteration complexity and improved convergence rate.  We investigate this tradeoff as well as other computational issues in the following sections.
\end{remark}

To completely justify the strategy of batching + weighting, we must also take into account the precomputation cost in computing the weighted distribution \eqref{newweights}, which increases with the batch size $b$.   In the next section, we refine Theorem \ref{thm:batchweight} precisely this way in the case of the least squares objective, where we can quantify more precisely the gain achieved by weighting and batching.  We give several explicit bounds and sampling strategies on the Lipschitz constants in this case that can be used for computationally efficient sampling.

\subsection{Least Squares Objective}

Although there are of course many methods for solving linear systems, methods like SGD for least squares problems have attracted recent attention due to their ability to utilize small memory footprints even for very large systems. In settings for example where the matrix is too large to be stored in memory, iterative approaches like the Kaczmarz method (a variant of SGD) are necessary.  With this motivation, we spend this section analyzing the least squares problem using weights and batching.

Consider the least squares objective 
$$
F(\bx) = \frac{1}{2} \| \A \bx - \bb \|_2^2 = \frac{1}{n} \sum_{i=1}^n f_i(\bx),
$$
where $f_i(\bx) =  \frac{n}{2} ( b_i - \langle \ba_i, \bx \rangle )^2$.  We assume the matrix $\A$ has full column-rank, so that there is a unique minimizer $\bx_{*}$ to the least squares problem:
$$
\xls = \bx_{*} = \arg \min_{\bx}  \| \A \bx - \bb \|_2^2.
$$
Note that the convexity and smoothness conditions are satisfied for such functions.  Indeed, observe that $\nabla f_i(\bx) = n (\langle \ba_i, \bx\rangle - b_i )\ba_i$, and
\begin{enumerate}
\item The individual Lipschitz constants are bounded by $L_i = n \| \ba_i \|_2^2$, and the average Lipschitz constant by 
$\frac{1}{n} \sum_i L_i = \| \bA \|_F^2$ (where $\|\cdot\|_F$ denotes the Frobenius norm),
\item The strong convexity parameter is $\mu = \sigma_{\min}^{-1}(\bA)$,  the reciprocal of the smallest singular value of $\bA$,
\item The residual is $\sigma^2 = n \sum_i \| \ba_i \|_2^2 | \langle \ba_i, \bx_{*} \rangle - b_i |^2$.
\end{enumerate}

In the batched setting, we compute
\begin{align}
g_{\tau_i}(\bx) &= \frac{1}{b}\sum_{k\in\tau_i} f_k(\bx) = \frac{n}{2b}\sum_{k\in\tau_i} (b_k - \langle\ba_k, \bx\rangle)^2 = \frac{d}{2} \| \bA_{\tau_i} \bx - \bb_{\tau_i} \|_2^2,
\end{align}
where we have written $\bA_{\tau_i}$ to denote the submatrix of $\bA$ consisting of the rows indexed by $\tau_i$.

Denote by $\sigma_{\tau}^2$ the residual in the batched setting. 
Since $\nabla g_{\tau_i}(\bx) = d \sum_{k \in \tau_i} (\langle\ba_k, \bx\rangle - b_k )\ba_k$, 
\begin{align}
\label{sigma:linear}
\sigma^2_{\tau} &= \frac{1}{d} \sum_{i=1}^d \| \nabla g_{\tau_i}(\bx_{*}) \|_2^2 = d \sum_{i=1}^d \| \sum_{k\in\tau_i} (\langle\ba_k, \bx_{*}\rangle - b_k)\ba_k \|_2^2 \nonumber \\
&= d \sum_{i=1}^d \| \bA^{*}_{\tau_i} ( \bA_{\tau_i} \bx_{*} - \bb_{\tau_i}) \|_2^2 \leq d  \sum_{i=1}^d  \| \bA_{\tau_i} \|^2 \| \bA_{\tau_i} \bx_{*} - \bb_{\tau_i} \|_2^2, \nonumber
\end{align}
where we have written $\|\bB\|$ to denote the spectral norm of the matrix $\bB$, and $\bB^*$ the adjoint of the matrix.  
Denote by $L_{\tau_i}$ the Lipschitz constant of $\nabla g_{\tau_i} $. 
Then we also have
\begin{align*}
L_{\tau_i} &= \sup_{\bx, \by} \frac{\|\nabla g_{\tau_i}(\bx) - \nabla g_{\tau_i}(\by)\|_2}{\|\bx-\by\|_2}\\
&= \frac{n}{b} \sup_{\bx, \by} \frac{\|\sum_{k\in\tau_i}\left[ (\langle\ba_k, \bx\rangle - b_k)\ba_k  - (\langle\ba_k, \by\rangle - b_k)\ba_k \right]\|_2}{\|\bx-\by\|_2}\\
&= \frac{n}{b} \sup_{\bz}\frac{\|\sum_{k\in\tau_i}\langle\ba_k, \bz\rangle\ba_k  \|_2}{\|\bz\|_2}\\
&= \frac{n}{b} \sup_{\bz}\frac{\|\bA_{\tau_i}^*\bA_{\tau_i}\bz  \|_2}{\|\bz\|_2}\\
&= \frac{n}{b} \|\bA_{\tau_i}^*\bA_{\tau_i}\| \\
&= d \| \bA_{\tau_i} \|^2.
\end{align*}

We see thus that \emph{if there exists a partition such that $\| \bA_{\tau_i} \|$ are as small as possible (e.g. within a constant factor of the row norms) for all $\tau_i$ in the partition, then both $\sigma^2_{\tau}$ and $L_{\tau} = \frac{1}{d} \sum_i L_{\tau_i}$ are decreased by a factor of the batch size $b$ compared to the unbatched setting.}  These observations are summed up in the following corollary of Theorem \ref{thm:batchweight} for the least squares case. 

\begin{corollary}
\label{LS:batchweight}
Consider $F(\bx) = \frac{1}{2} \| \A \bx - \bb \|_2^2 = \frac{1}{2} \sum_{i=1}^{d} \| \bA_{\tau_i} \bx - \bb_{\tau_i} \|_2^2$.  Consider the batched weighted SGD iteration 
\begin{equation}
\label{SGD:LS}
\bx_{k+1} \leftarrow \bx_{k} - \frac{\gamma}{p(\tau_i)} \sum_{j \in \tau_i} ( \langle\ba_j, \bx_k\rangle - b_j)\ba_j,
\end{equation}
with weights 
\begin{equation}
\label{w:batchweights}
p(\tau_i) = \frac{b}{2n} + \frac{1}{2} \cdot \frac{\| \bA_{\tau_i} \|^2}{ \sum_{i=1}^{d} \| \bA_{\tau_i} \|^2 }.
\end{equation}
For any desired $\varepsilon$, and using a stepsize of  
\begin{equation}\label{LSstep}
\gamma = \frac{ \frac{1}{4}\varepsilon}{\varepsilon \sum_{i=1}^d \| \bA_{\tau_i} \|^2 + d \sigma_{\min}^{-2}(\bA)  \sum_{i=1}^d  \| \bA_{\tau_i} \|^2 \| \bA_{\tau_i} \bx_{*} - \bb_{\tau_i} \|_2^2},
\end{equation}
we have that after
\begin{equation}
\label{k:batchweight}
k = \left\lceil4 \log(2 \varepsilon_0/\varepsilon) \left(  \sigma_{\min}^{-2}(\bA) \sum_{i=1}^d \| \bA_{\tau_i} \|^2  + \frac{d \sigma_{\min}^{-4}(\bA)  \sum_{i=1}^d  \| \bA_{\tau_i} \|^2 \| \bA_{\tau_i} \bx_{*} - \bb_{\tau_i} \|_2^2}{ \varepsilon}  \right)\right\rceil
\end{equation}
iterations of \eqref{SGD:LS}, $\mathbb{E}^{(p)} \| {\bf x}_k - {\bf x}_{*} \|_2^2 \leq \varepsilon$ where $\mathbb{E}^{(p)}[\cdot ]$ means the expectation with respect to the index at each iteration drawn according to the weighted distribution \eqref{w:batchweights}. 
\end{corollary}

This corollary suggests a heuristic for batching and weighting in SGD for least squares problems, in order to optimize the convergence rate. Note of course that, like other similar results for SGD, it is only a heuristic since in particular to compute the step size and required number of iterations in \eqref{LSstep} and \eqref{k:batchweight}, one needs an estimate on the size of the system residual $\|\bA\bx_\star-\bb\|_2$ (which is of course zero in the consistent case).  We summarize the desired procedure here:

\begin{enumerate}
\item Find a partition $\tau_1, \tau_2, \dots, \tau_d$ that roughly minimizes $\sum_{i=1}^{d} \| \bA_{\tau_i} \|^2$ among all such partitions.
\item Apply the weighted SGD algorithm \eqref{SGD:LS} using weights given by \eqref{w:batchweights}.
%$$
%p(\tau_i) = \frac{1}{2d} + \frac{1}{2} \cdot \frac{\| \bA_{\tau_i} \|_2^2}{\sum_{i=1}^{d} \| \bA_{\tau_i} \|_2^2}.
%$$
\end{enumerate}

We can compare the results of Corollary \ref{LS:batchweight} to the results for weighted SGD when a single functional is selected in each iteration, where the number of iterations to achieve expected error $\varepsilon$ is 
\begin{equation}
\label{k:standard}
k = \left\lceil4 \log(2 \varepsilon_0/\varepsilon) \left(  \sigma_{\min}^{-2}(\bA) \sum_{i=1}^n \| \ba_i  \|^2  + \frac{n \sigma_{\min}^{-4}(\bA)  \sum_{i=1}^n  \| \ba_{i} \|^2 \| \langle  \ba_{i}, \bx_{*} \rangle - b_i \|_2^2}{ \varepsilon}  \right)\right\rceil.
\end{equation}
That is, the ratio between the standard weighted number of iterations $k_{stand}$ in \eqref{k:standard} and the batched weighted number of iterations $k_{batch}$ in \eqref{k:batchweight} is
\begin{equation}
\label{k:ratio}
\frac{k_{stand}}{k_{batch}} = \frac{\varepsilon \sum_{i=1}^n \| \ba_i  \|^2  + n \sigma_{\min}^{-2}(\bA)  \sum_{i=1}^n  \| \ba_{i} \|^2 \| \langle  \ba_{i}, \bx_{*} \rangle - b_i \|_2^2}{ \varepsilon \sum_{i=1}^d \| \bA_{\tau_i} \|^2  + d \sigma_{\min}^{-2}(\bA)  \sum_{i=1}^d  \| \bA_{\tau_i} \|^2 \| \bA_{\tau_i} \bx_{*} - \bb_{\tau_i} \|_2^2}.
\end{equation}
In case the least squares residual error is uniformly distributed over the $n$ indices, that is, $ \| \langle  \ba_{i}, \bx_{*} \rangle - b_i \|_2^2 \approx \frac{1}{n} \| \bA \bx_{*} - \bb \|^2$ for each $i \in [n]$, this factor reduces to 
\begin{equation}
\label{k:UNIFORMratio}
\frac{k_{stand}}{k_{batch}} = \frac{  \| \bA  \|_F^2}{ \sum_{i=1}^d \| \bA_{\tau_i} \|^2}.
\end{equation}
It follows thus that the combination of batching and weighting in this setting always reduces the iteration complexity compared to weighting alone, and can result in up to a factor of $b$ speedup:
$$
1 \leq \frac{k_{stand}}{k_{batch}}  \leq b.
$$
In the remainder of this section, we consider several families of matrices where the maximal speedup is achieved, $\frac{k_{stand}}{k_{batch}}  \approx b$.     We also take into account the computational cost of computing the norms $\| \bA_{\tau_i} \|^2$ which determine the weighted sampling strategy.

\begin{description}
\item[\textbf{Orthonormal systems}]  It is clear that the advantage of mini-batching is strongest when the rows of $\bA$ in each batch are orthonormal.  In the extreme case where $\bA$ has orthonormal rows, we have
\begin{align}\label{ortho}
\overline{L}_\tau &= \sum_{i=1}^{d} \|\bA_{\tau_i}^*\bA_{\tau_i}\|  
= \frac{n}{b} = \frac{1}{b} \overline{L}.
\end{align}
Thus for orthonormal systems, we gain a factor of $b$ by using  mini-batches of size $b$.  However, there is little advantage to weighting in this case as all Lipschitz constants are the same. \\ 

\item[\textbf{Incoherent systems}] More generally, the advantage of mini-batching is strong when the rows $\ba_i$ within any particular batch are \emph{nearly} orthogonal.  Suppose that each of the batches is well-conditioned in the sense that 
\begin{equation}
\label{eq:incoherence}
\sum_{i=1}^n \| \ba_i \|_2^2 \geq C' n, \quad \quad \|\bA_{\tau_i}^*\bA_{\tau_i}\| = \|\bA_{\tau_i}\bA_{\tau_i}^*\|   \leq C, \quad \quad i = 1, \dots, d,
\end{equation}
For example, if $\bA^{*}$ has the \textit{restricted isometry property} \cite{RefWorks:48} of level $\delta$ at sparsity level $b$, \eqref{eq:incoherence} holds with $C \leq 1 + \delta$.  Alternatively, if $\bA$ has unit-norm rows and is incoherent, i.e. $ \max_{i \neq j} | \langle \ba_i, \ba_j \rangle | \leq \frac{\alpha}{b-1}$, then \eqref{eq:incoherence} holds with constant $C \leq 1 + \alpha$ by Gershgorin circle theorem.  

If the incoherence condition \eqref{eq:incoherence} holds, we gain a factor of $b$ by using weighted mini-batches of size $b$:
\begin{align*}
\overline{L}_\tau &= \sum_{i=1}^{d} \|\bA_{\tau_i}^*\bA_{\tau_i}\|  
\leq C \frac{n}{b}
\leq \frac{C}{C'} \frac{ \overline{L}}{b}.
\end{align*}

\item[\textbf{Incoherent systems, variable row norms}] More generally, consider the case where the rows of $\bA$ are nearly orthogonal to each other, but not normalized as in \eqref{eq:incoherence}.  We can then write $\bA = \bD \bPsi,$ where $\bD$ is an $n \times n$ diagonal matrix with entry $d_{ii} = \| \ba_i \|_2$, and $\bPsi$ with normalized rows satisfies
$$
\|\bPsi_{\tau_i}^*\bPsi_{\tau_i}\| = \|\bPsi_{\tau_i}\bPsi_{\tau_i}^*\|   \leq C, \quad \quad i = 1, \dots, d,
$$
as is the case if, e.g.,  $\bPsi$ has the restricted isometry property or $\bPsi$ is incoherent.  

\bigskip
In this case, we have
\begin{align}
\|\bA_{\tau_i}^*\bA_{\tau_i}\|  &= \|\bA_{\tau_i}\bA_{\tau_i}^*\|  =  \| \bD_{\tau_i} \bPsi_{\tau_i} \bPsi_{\tau_i}^{*} \bD_{\tau_i} \| \nonumber \\
&\leq   \max_{k \in \tau_i} \| \ba_k \|_2^2 \| \bPsi_{\tau_i} \bPsi_{\tau_i}^{*} \| \nonumber \\
&\leq  C \max_{k \in \tau_i} \| \ba_k \|_2^2, \quad \quad i = 1, \dots, d. \label{max-norm}
\end{align}
Thus, 
\begin{equation}%\label{max-norm}
\overline{L}_\tau = \sum_{i=1}^{d} \|\bA_{\tau_i}^*\bA_{\tau_i}\|  \leq C \sum_{i=1}^{d}  \max_{k \in \tau_i} \| \ba_k \|_2^2.
\end{equation}
In order to minimize the expression on the right hand side over all partitions into blocks of size $b$, we partition the rows of $\bA$ according to the order of the decreasing rearrangement of their row norms.  This batching strategy results in a factor of $b$ gain in iteration complexity compared to weighting without batching:
\begin{align}\label{max-norm-bound}
\overline{L}_\tau &\leq C \sum_{i=1}^{d} \| \ba_{((i-1)b +1)} \|_2^2 \nonumber \\
&\leq \frac{C}{b-1} \sum_{i=1}^{n} \| \ba_{i} \|_2^2 \nonumber \\
&\leq  \frac{C'}{b} \overline{L}.
\end{align}
\end{description}

We now turn to the practicality of computing the distribution given by the constants $L_{\tau_i}$.  We propose several options to efficiently compute these values given the ability to parallelize over $b$ cores.

\begin{description}
\item[\textbf{Max-norm}] The discussion above suggests the use of the maximum row norm of a batch as a proxy for the Lipschitz constant.  Indeed, \eqref{max-norm} shows that the row norms give an upper bound on these constants.  Then, \eqref{max-norm-bound} shows that up to a constant factor, such a proxy still has the potential to lead to an increase in the convergence rate by a factor of $b$.  Of course, computing the maximum row norm of each batch costs on the order of $mn$ flops (the same as the non-batched weighted SGD case). 

\item[\textbf{Power method}] In some cases, we may utilize the power method to approximate $\|\bA_{\tau_i}^*\bA_{\tau_i}\|$ efficiently.  Suppose that for each batch we can approximate this quantity by $\hat{Q}_{\tau_i}$.  Classical results on the power method allow one to approximate the norm to within an arbitrary additive error, with a number of iterations that depends on the spectral gap of the matrix.   
An alternative approach, that we consider here, can be used to obtain approximations leading to a \textit{multiplicative} factor difference in the convergence rate, without dependence on the eigenvalue gaps $\lambda_1/\lambda_2$ within batches.  For example, \cite[Lemma 5]{klein1996efficient} shows that with high probability with respect to a randomized initial direction to the power method, after $T \geq \varepsilon^{-1}\log(\varepsilon^{-1}b)$ iterations of the power method, one can guarantee that 
$$
\|\bA_{\tau_i}^*\bA_{\tau_i}\| \geq \hat{Q}_{\tau_i} \geq \frac{\|\bA_{\tau_i}^*\bA_{\tau_i}\|}{1+\varepsilon}.
$$
At $b^2$ computations per iteration of the power method, the total computational cost (to compute all quantities in the partition), shared over all $b$ cores, is $b\varepsilon^{-1}\log(\varepsilon^{-1}\log(b))$.  This is actually potentially much \textit{lower} than the cost to compute all row norms $L_i = \|\ba_i\|_2^2$ as in the standard non-batched weighted method. 
In this case, the power method yields 
$$
\overline{L}_\tau \geq \frac{b}{n} \sum_{i=1}^{d} \frac{n}{b}\hat{Q}_{\tau_i} \geq \frac{\overline{L}_\tau}{1+\varepsilon},
$$
for a constant $\varepsilon$.

\end{description}

\section{Mini-batch SGD with weighting: the non-smooth case}\label{sec:non-smooth}

We next present analogous results to the previous section for objectives which are strongly convex but lack the smoothness assumption.  Like the least squares objective in the previous section, our motivating example here will be the support vector machine (SVM) with hinge loss objective.  

A classical result (see e.g. \cite{nes04,shamir2012stochastic,rakhlin2011making}) for SGD establishes a convergence bound of SGD with non-smooth objectives.  In this case, rather than taking a step in the gradient direction of a functional, we move in a direction of a subgradient.  Instead of utilizing the Lipschitz constants of the gradient terms, we utilize the Lipschitz constants of the actual functionals themselves. Note that in the non-smooth case one cannot guarantee convergence of the iterates $\bx_k$ to a unique minimizer $\bx_\star$ so instead one seeks convergence of the objective value itself. Concretely, a classical bound is of the following form.

\begin{proposition}\label{propNS}
 Let the objective $F(\bx) = \mathbb{E}g_i(\bx)$ with minimizer $\bx_\star$ be a $\mu$-strongly convex (possibly non-smooth) objective.  Run SGD using a subgradient $h_i$ of a randomly selected functional $g_i$ at each iteration.  Assume that $\mathbb{E}h_i \in \partial F(\bx_k)$ (expectation over the selection of subgradient $h_i$) and that
$$
\max_{\bx, \by} \frac{\| g_i(\bx) - g_i(\by) \|}{\| \bx - \by \|} \leq \max_{\bx} \|  h_i(\bx) \| \leq G_i.
$$
Set $\overline{G^2} = \mathbb{E}(G_i^2)$.  Using step size $\gamma = \gamma_k = 1/(\mu k)$, we have 
\begin{equation}\label{nssgd}
\mathbb{E} \left[F(\bx_k) - F(\bx_\star)\right] \leq \frac{C\overline{G^2} (1+\log k)}{\mu k},
\end{equation}
where $C$ is an absolute constant.
\end{proposition}

Such a result can be improved by utilizing averaging of the iterations; for example, if $\bx_k^\alpha$ denotes the average of the last $\alpha k$ iterates, then the convergence rate bound \eqref{nssgd} can be improved to:

$$
\mathbb{E} \left[F(\bx_k) - F(\bx_\star)\right] \leq \frac{C\overline{G^2} \left(1+\log \frac{1}{\min(\alpha, (1+1/k)-\alpha)}\right)}{\mu k} \leq \frac{C\overline{G^2} \left(1+\log \frac{1}{\min(\alpha, 1-\alpha)}\right)}{\mu k}.
$$

Setting $m_\alpha = 1+\log \frac{1}{\min(\alpha, 1-\alpha)}$, we see that to obtain an accuracy of $\mathbb{E} \left[F(\bx_k) - F(\bx_\star)\right] \leq \varepsilon$, it suffices that
$$
k \geq \frac{C\overline{G^2}m_\alpha}{\mu\varepsilon}.
$$

 In either case, it is important to notice the dependence on $\overline{G^2} = \mathbb{E}(G_i^2)$.  By using weighted sampling with weights $p(i) = G_i / \sum_i G_i$, we can improve this dependence to one on $(\overline{G})^2$, where $\overline{G} = \mathbb{E}G_i$ \cite{needell2014stochastic,zhao2014stochastic}.  Since $\overline{G^2} - (\overline{G})^2 = $ Var$(G_i)$, this improvement reduces the dependence by an amount equal to the variance of the Lipschitz constants $G_i$.  Like in the smooth case, we now consider not only weighting the distribution, but also by batching the functionals $g_i$.  This yields the following result, which we analyze for the specific instance of SVM with hinge loss below.

\begin{theorem}
\label{thm:batchweightNS}
Instate the assumptions and notation of Proposition \ref{propNS}.  Consider the $d = n/b$ batches $g_{\tau_i}(\bx) = \frac{1}{b}\sum_{j\in\tau_i} g_j(\bx)$, and assume each batch $g_{\tau_i}$ has Lipschitz constant $G_{\tau_i}$.  Write $\overline{G}_\tau = \mathbb{E}G_{\tau_i}$.  Run the weighted batched SGD method with averaging as described above, with step size $\gamma / p(\tau_i)$. 
%$$
%\bx_{k+1} \leftarrow \bx_{k} - \frac{\gamma}{d p(\tau_i)} \nabla g_{\tau_i}(\bx_k)
%$$
For any desired $\varepsilon$, %and using a stepsize of 
%$$
%\gamma = \frac{\mu \varepsilon}{4 \varepsilon \mu \overline{L}_{\tau}+ \sigma_{\tau}^2},
%$$
it holds that after
$$
k = \frac{C(\overline{G}_\tau)^2 m_\alpha}{\mu\varepsilon}
$$
iterations with weights 
\begin{equation}\label{newweightsNS}
%\label{w:batchweights}
p(\tau_i) = \frac{G_{\tau_i}}{\sum_j{G}_{\tau_j}},
\end{equation}
we have $\mathbb{E}^{(p)} \left[F(\bx_k) - F(\bx_\star)\right] \leq \varepsilon$ where $\mathbb{E}^{(p)}[\cdot ]$ means the expectation with respect to the index at each iteration drawn according to the weighted distribution \eqref{newweightsNS}. 
\end{theorem}
\begin{proof}
 Applying weighted SGD with weights $p(\tau_i)$, we re-write the objective $F(\bx) = \mathbb{E} \left( g_i(\bx) \right)$ as 
$F(\bx) = \mathbb{E}^{(p)}  \left(\hat{g}_{\tau_i}(\bx) \right)$,  where 
$$
\hat{g}_{\tau_i}(x) =\left(\frac{1}{n}\sum_{j} G_{\tau_j}\right)  \left( \frac{1}{G_{\tau_i}}\sum_{j\in\tau_i}g_j(\bx) \right) = 
\left(\frac{b}{n}\sum_{j} G_{\tau_j}\right)  \left( \frac{g_{\tau_i}(\bx)}{G_{\tau_i}} \right).
$$

Then, the Lipschitz constant $\hat{G}_i$ of $\hat{g}_{\tau_i}$ is bounded above by $\hat{G}_i = \frac{b}{n}\sum_{j} G_{\tau_j}$, and so 
$$\mathbb{E}^{(p)} \hat{G}_i^2 = \sum_i \frac{G_{\tau_i}}{\sum_j{G}_{\tau_j}} \left(\frac{b}{n}\sum_{j} G_{\tau_j} \right)^2 = \left(\frac{b}{n}\sum_{j} G_{\tau_j} \right)^2 = (\mathbb{E} G_{\tau_i} )^2 = (\overline{G}_\tau)^2.$$  
The result follows from an application of Proposition \ref{propNS}.
\qed
\end{proof}

We now formalize these bounds and weights for the SVM with hinge loss objective.  Other objectives such as L1 regression could also be adapted in a similar fashion, e.g. utilizing an approach as in \cite{yang2016weighted}.

\subsection{SVM with Hinge Loss} 

We now consider the SVM with hinge loss problem as a motivating example for using batched weighted SGD for non-smooth objectives. 
 Recall the SVM with hinge loss objective is 
\begin{equation}\label{obj_hinge}
F(\bx) := \frac{1}{n}\sum_{i=1}^n [y_i \langle \bx, \ba_i \rangle ]_{+} + \frac{\lambda}{2} \| \bx \|_2^2 = \mathbb{E} g_i(\bx),
\end{equation}
where $y_i \in \{ \pm 1 \}$, $[u]_{+} = \max(0, u)$, and
$$
g_i(\bx) =  [y_i \langle \bx, \ba_i \rangle ]_{+} + \frac{\lambda}{2} \| \bx \|_2^2.
$$
This is a key example where the components are ($\lambda$-strongly) convex but no longer smooth.  Still, each $g_i$ has a well-defined subgradient:
$$
\nabla g_i(\bx) = \chi_i(\bx) y_i \ba_i + \lambda\bx,
$$
where $ \chi_i(\bx) = 1$ if $y_i  \langle \bx, \ba_i \rangle < 1$ and $0$ otherwise.  %We can consider the \emph{proximal} SGD algorithm with weighted sampling \cite{zhao2014stochastic}.  
It follows that $g_i$ is Lipschitz and its Lipschitz constant is bounded by
$$
G_i := \max_{\bx, \by} \frac{\| g_i(\bx) - g_i(\by) \|}{\| \bx - \by \|} \leq \max_{\bx} \| \nabla g_i(\bx) \| \leq \| \ba_i \|_2 +\lambda.
$$
As shown in \cite{zhao2014stochastic},  \cite{needell2014stochastic},  in the setting of non-smooth objectives of the form \eqref{obj_hinge}, where  the components are not necessarily smooth, but each $g_i$ is $G_i$-Lipschitz, the performance of SGD depends on the quantity 
$\overline{G^2} = \mathbb{E}[G_i^2]$.   In particular, the iteration complexity depends linearly on $\overline{G^2}$.   

\bigskip

For the hinge loss example, we have calculated that
$$
\overline{G^2} = \frac{1}{n} \sum_{i=1}^n \left(\| \ba_i \|_2 + \lambda\right)^2 \leq 2\lambda^2 + \frac{2}{n}\sum_{i=1}^n \| \ba_i \|_2^2 .
$$
Incorporating (non-batch) weighting to this setting, as discussed in  \cite{needell2014stochastic}, reduces the iteration complexity to depend linearly on $(\overline{G})^2 = (\mathbb{E}[G_i])^2$, which is at most $\overline{G^2}$ and can be as small as $\frac{1}{n} \overline{G^2}$.  For the hinge loss example, we have
$$
(\overline{G})^2 =  \left( \lambda + \frac{1}{n} \sum_{i=1}^n \| \ba_i \|_2 \right)^2 .
$$
We note here that one can incorporate the dependence on the regularizer term $\frac{\lambda}{2}\|\bx\|_2^2$ in a more optimal way by bounding the functional norm only over the iterates themselves, as in \cite{takavc2013mini,rakhlin2011making}; however, we choose a crude upper bound on the Lipschitz constant here in order to maintain a dependence on the \textit{average} constant rather than the \textit{maximum}, and only sacrifice a constant factor.
\subsubsection{Batched sampling}

The paper \cite{takavc2013mini} considered batched SGD for the hinge loss objective.  For batches $\tau_i$ of size $b$, 
let $g_{\tau_i}= \frac{\lambda}{2} \| \bx \|_2^2 + \frac{1}{b} \sum_{k \in \tau_i} [y_k \langle \bx, \ba_k \rangle ]_{+}$ and observe
$$
F(\bx) := \frac{1}{n}\sum_{i=1}^n [y_i \langle \bx, \ba_i \rangle ]_{+} + \frac{\lambda}{2} \| \bx \|_2^2 = \mathbb{E} g_{\tau_i}(\bx).
$$
We now bound the Lipschitz constant $G_\tau$ for a batch.  Let $\chi = \chi_k(\bx)$ and $\bA_{\tau}$ have rows $y_k \ba_k$ for $k \in \tau$.   We have
\begin{align}
%G_{\tau} &\leq 
\max_{\bx} \left\| \frac{1}{b} \sum_{k \in \tau_i} \chi_k(\bx) y_k \ba_k  \right\|_2 %\nonumber \\
&=  \max_{\bx} \sqrt{  \left\langle \frac{1}{b} \sum_{k \in \tau_i} \chi_k(\bx) y_k \ba_k,  \frac{1}{b} \sum_{k \in \tau_i} \chi_k(\bx) y_k \ba_k \right\rangle }  \nonumber \\
&= \frac{1}{b}  \max_{\bx} \sqrt{ \chi^* \bA_{\tau} \bA_{\tau}^{*}  \chi}  \nonumber \\
%&\leq \frac{1}{b}  \max_{\bx} \sqrt{ \chi^* \bA_{\tau} \bA_{\tau}^{*} \chi}  \nonumber \\
&\leq \frac{1}{b}  \sqrt{ b \| \bA_{\tau} \bA_{\tau}^{*} \| }  \nonumber \\
&=\frac{1}{\sqrt{b}} \| \bA_{\tau} \|, 
\end{align}
and therefore $G_{\tau} \leq \frac{1}{\sqrt{b}} \| \bA_{\tau} \| + \lambda$.
Thus, for batched SGD without weights, the iteration complexity depends linearly on 
\begin{align}
\overline{G_{\tau}^2} &= \frac{b}{n} \sum_{i=1}^{d} G_{\tau_i}^2 \nonumber \\
&\leq 2\lambda^2 + \frac{2}{n} \sum_{i=1}^{d} \| \bA_{\tau_i} \|^2 \nonumber \\
&= 2\lambda^2 + \frac{2}{n} \sum_{i=1}^{d} \| \bA_{\tau_i}^{*} \bA_{\tau_i} \|. \nonumber 
\end{align}
Even without weighting, we already see potential for drastic improvements, as noted in \cite{takavc2013mini}.  For example, in the orthonormal case, where $\| \bA_{\tau_i}^{*} \bA_{\tau_i} \| = 1$ for each $\tau_i$, we see that with appropriately chosen $\lambda$, $\overline{G_{\tau}^2}$ is on the order of $\frac{1}{b}$, which is a factor of $b$ times smaller than $\overline{G^2} \approx 1$.  Similar factors are gained for the incoherent case as well, as in the smooth setting discussed above.  Of course, we expect even more gains by utilizing both batching and weighting.

\subsubsection{Weighted batched sampling}
Incorporating weighted batched sampling, where we sample batch $\tau_i$ with probability proportional to $G_{\tau_i},$ the iteration complexity is reduced to a linear dependence on $(\overline{G_{\tau}})^2$, as in Theorem \ref{thm:batchweightNS}.  For hinge loss, we calculate
\begin{align*}
(\overline{G_{\tau}})^2 &= \left( \frac{b}{n} \sum_{i=1}^{d} G_{\tau_i} \right)^2  \leq \left( \frac{b}{n} \sum_{i=1}^{d} \frac{1}{\sqrt{b}} \| \bA_{\tau_i} \|  +\lambda\right)^2 = \left( \lambda + \frac{\sqrt{b}}{n} \sum_{i=1}^{d} \| \bA_{\tau_i} \|  \right)^2.
\end{align*}

We thus have the following guarantee for the hinge loss objective.

\begin{corollary}
\label{HL:batchweight}
%\notate{Change this to use Corollary \ref{corNS} with extra $\lambda^2$ term??}
Instate the notation of Theorem \ref{thm:batchweightNS}.
Consider $F(\bx) = \frac{1}{n}\sum_{i=1}^n [y_i \langle \bx, \ba_i \rangle ]_{+} + \frac{\lambda}{2} \| \bx \|_2^2$.  Consider the batched weighted SGD iteration
\begin{equation}
\label{SGD:HL}
\bx_{k+1} \leftarrow \bx_{k} - \frac{1}{\mu k p(\tau_i)} \left( \lambda\bx_k + \frac{1}{b}\sum_{j\in\tau_i}\chi_j(\bx_k)y_j\ba_j\right),
\end{equation}
%\notate{Why do ZZ and others use step size without the p(i) term??}
where $ \chi_j(\bx) = 1$ if $y_j  \langle \bx, \ba_j \rangle < 1$ and $0$ otherwise. Let $\bA_{\tau}$ have rows $y_j \ba_j$ for $j \in \tau$.
For any desired $\varepsilon$, 
we have that after
\begin{equation}
\label{k:batchweightHL}
k = \frac{Cm_\alpha\left( \lambda + \frac{\sqrt{b}}{n} \sum_{i=1}^{d} \| \bA_{\tau_i} \|  \right)^2}{\lambda\varepsilon}
\end{equation}
iterations of \eqref{SGD:HL} with weights 
\begin{equation}
\label{w:batchweightsHL}
p(\tau_i) = \frac{\|\bA_{\tau_i}\| + \lambda\sqrt{b}}{\frac{n}{\sqrt{b}}\lambda + \sum_j\|\bA_{\tau_j}\|},
\end{equation}
it holds that $\mathbb{E}^{(p)} [ F({\bf x}_k) - F({\bf x}_{*}) ] \leq \varepsilon$.
\end{corollary}

\section{Experiments}\label{sec:exps}

In this section we present some simple experimental examples that illustrate the potential of utilizing weighted mini-batching.  We consider several test cases as illustration.

\begin{description}
\item[\textbf{Gaussian linear systems}] The first case solves a linear system $\bA\bx=\bb$, where $\bA$ is a matrix with i.i.d. standard normal entries (as is $\bx$, and $\bb$ is their product).  In this case, we expect the Lipschitz constants of each block to be comparable, so the effect of weighting should be modest.  However, the effect of mini-batching in parallel of course still appears.  Indeed, Figure \ref{fig1} (left) displays the convergence rates in terms of iterations for various batch sizes, where each batch is selected with probability as in \eqref{w:batchweights}.  When batch updates can be run in parallel, we expect the convergence behavior to mimic this plot (which displays iterations).  We see that in this case, larger batches yield faster convergence.  In these simulations, the step size $\gamma$ was set as in \eqref{LSstep} (approximations for Lipschitz constants also apply to the step size computation) for the weighted cases and set to the optimal step size as in \cite[Corollary 3.2]{needell2014stochastic} for the uniform cases.  Behavior using uniform selection is very similar (not shown), as expected in this case since the Lipschitz constants are roughly constant.   %Figure \ref{fig1} (right) shows similar, although slightly worse, convergence behavior when the batches are selected uniformly at random, as is expected in this case.  %
Figure \ref{fig1} (right) highlights the improvements in our proposed weighted batched SGD method versus the classical, single functional and unweighted, SGD method.  The power method refers to the method discussed at the end of Section \ref{sec:smooth}, and max-norm method refers to the approximation using the maximum row norm in a batch, as in \eqref{max-norm}.  The notation ``(opt)'' signifies that the optimal step size was used, rather than the approximation; otherwise in all cases both the sampling probabilities \eqref{w:batchweights} and step sizes \eqref{LSstep} were approximated using the approximation scheme given.   Not suprisingly, using large batch sizes yields significant speedup. 

\item[\textbf{Gaussian linear systems with variation}] We next test systems that have more variation in the distribution of Lipschitz constants.  We construct a matrix $\bA$ of the same size as above, but whose entries in the $k$th row are i.i.d. normally distributed with mean zero and variance $k^2$.   We now expect a large effect both from batching and from weighting.  In our first experiment, we select the fixed batches randomly at the onset, and compute the probabilities according to the Lipschitz constants of those randomly selected batches, as in \eqref{w:batchweights}.  The results are displayed in the left plot of Figure \ref{fig4}.  In the second experiment, we batch sequentially, so that rows with similar Lipschitz constants (row norms) appear in the same batch, and again utilized the weighted sampling.  The results are displayed in the center plot of Figure \ref{fig4}.  Finally, the right plot of Figure \ref{fig4} shows convergence when batching sequentially and then employing uniform (unweighted) sampling.  As our theoretical results predict, batching sequentially yields better convergence, as does utilizing weighted sampling.  

Since this type of system nicely highlights the effects of both weighting and batching, we performed additional experiments using this type of system.  Figure \ref{fig5} highlights the improvements gained by using weighting.  In the first plot, we see that for all batch sizes improvements are obtained by using weighting, even more so than in the standard normal case, as expected (note that we cut the curves off when the weighted approach reaches machine precision).  In the bottom plot, we see that the number of iterations to reach a desired threshold is also less using the various weighting schemes; we compare the sampling method using exact computations of the Lipshitz constants (spectral norms), using the maximum row norm as an approximation as in \eqref{max-norm}, and using the power method (using number of iterations equal to $\epsilon^{-1}\log(\epsilon^{-1}b)$ with $\epsilon=0.01$).  Step size $\gamma$ used on each batch was again set as in \eqref{LSstep} (approximations for Lipschitz constants also apply to the step size computation) for the weighted cases and as in \cite[Corollary 3.2]{needell2014stochastic} for the uniform cases.  For cases when the exact step size computation was used rather than the corresponding approximation, we write ``(opt)''.  For example, the marker ``Max norm (opt)'' represents the case when we use the maximum row norm in the batch to approximate the Lipschitz constant, but still use the exact spectral norm when computing the optimal step size.  This of course is not practical, but we include these for demonstration.  Figure \ref{fig6} highlights the effect of using batching.  The first plot confirms that larger batch sizes yield significant improvement in terms of L2-error and convergence (note that again all curves eventually converge to a straight line due to the error reaching machine precision).  The bottom plot highlights the improvements in our proposed weighted batched SGD methods versus the classical, single functional and unweighted, SGD method. 

We next further investigate the effect of using the power method to approximate the Lipschitz constants used for the probability of selecting a given batch.  We again create the batches sequentially and fix them throughout the remainder of the method.  At the onset of the method, after creating the batches, we run the power method using $\epsilon^{-1}\log(\epsilon^{-1}b)$ iterations (with $\epsilon=0.01$) per batch, where we assume the work can evenly be divided among the $b$ cores.  We then determine the number of computational flops required to reach a specified solution accuracy using various batch sizes $b$.  The results are displayed in Figure \ref{fig6b}.  The first plot shows the convergence of the method; comparing with the first plot of Figure \ref{fig4}, we see that the convergence is slightly slower than when using the precise Lipschitz constants, as expected.  The last plot of Figure \ref{fig6b} shows the number of computational flops required to achieve a specified accuracy, as a function of the batch size.  We see that there appears to be an ``optimal'' batch size, around $b=8$ for this case, at which the savings in computational time computing the Lipschitz constants and the additional iterations required due to the inaccuracy are balanced.  

\item[\textbf{Correlated linear systems}] We next tested the method on systems with correlated rows, using a matrix with i.i.d. entries uniformly distributed on $[0,1]$.  When the rows are correlated in this way, the matrix is poorly conditioned and thus convergence speed suffers.   Here, we are particularly interested in the behavior when the rows also have high variance; in this case, row $k$ has uniformly distributed entries on $[0, \sqrt{3}k]$ so that each entry has variance $k^2$ like the Gaussian case above.   Figure \ref{figcor} displays the convergence results when creating the batches randomly and using weighting (left), creating the batches sequentially and using weighting (center), and creating the batches sequentially and using unweighted sampling (right).  Like Figure \ref{fig4}, we again see that batching the rows with larger row norms together and then using weighted sampling produces a speedup in convergence. %Note that using Kaczmarz-style step sizes in this case yields divergence for all batch sizes greater than 2. \notate{Why?}

\item[\textbf{Orthonormal systems}] As mentioned above, we expect the most notable improvement in the case when $\bA$ is an orthonormal matrix.  For this case, we run the method on a $200\times 200$ orthonormal discrete Fourier transform (DFT) matrix.  As seen in the first plot of Figure \ref{fig8}, we do indeed see significant improvements in convergence with batches in our weighted scheme. Of course, if the matrix is orthonormal one could also simply apply $\bA^*$ to solve the system, but we include these experiments for intuition and comparison.

\item[\textbf{Sparse systems}] Lastly, we show convergence for the batched weighted scheme on sparse Gaussian systems. The matrix is generated to have $20\%$ non-zero entries, and each non-zero entry is i.i.d. standard normal.   Figure \ref{fig8} (center) shows the convergence results.  The convergence behavior is similar to the non-sparse case, as expected, since our method does not utilize any sparse structure. 

\item[\textbf{Tomography data}] The final system we consider is a real system from tomography.  The system was generated using the Matlab Regularization Toolbox by P.C. Hansen (\url{http://www.imm.dtu.dk/~pcha/Regutools/}) \cite{hansen2007regularization}.  This creates a 2D tomography problem $\bA\bx= \bb$ for an $n\times d$ matrix with $n=fN^2$ and $d=N^2$, where $\bA$ corresponds to the absorption along a random line through an $N\times N$ grid.  We set $N=20$ and the oversampling factor $f=3$.  Figure \ref{fig8} (right) shows the convergence results. 

\item[\textbf{Noisy (inconsistent) systems}]  Lastly, we consider systems that are noisy, i.e. they have no exact solution.  We seek convergence to the least squares solution $\xls$.  We consider the same Gaussian matrix with variation as desribed above.  We first generate a consistent system $\bA\bx=\bb$ and then add a residual vector $\be$ to $\bb$ that has norm one, $\|\be\|_2 = 1$.  Since the step size in \eqref{LSstep} depends on the magnitude of the residual, it will have to be estimated in practice.  In our experiments, we estimate this term by an upper bound which is $1.1$ times larger in magnitude than the true residual $\|\bA\xls - \bb\|_2$.  In addition, we choose an accuracy tolerance of $\varepsilon = 0.1$.  
Not surprisingly, our experiments in this case show similar behavior to those mentioned above, only the method convergences to a larger error (which can be lowered by adjusting the choice of $\varepsilon$).  An example of such results in the correlated Gaussian case are shown in Figure \ref{fignoise}.  

\end{description}

\begin{figure}[!ht]
\includegraphics[width=3.3in]{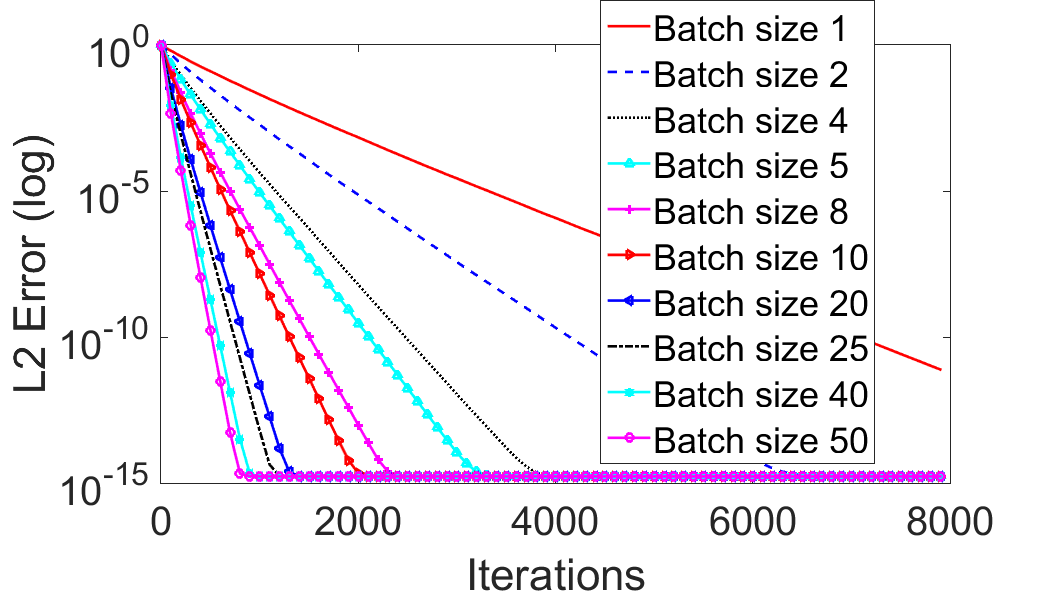} \hspace{-0.2in}\includegraphics[width=3.3in]{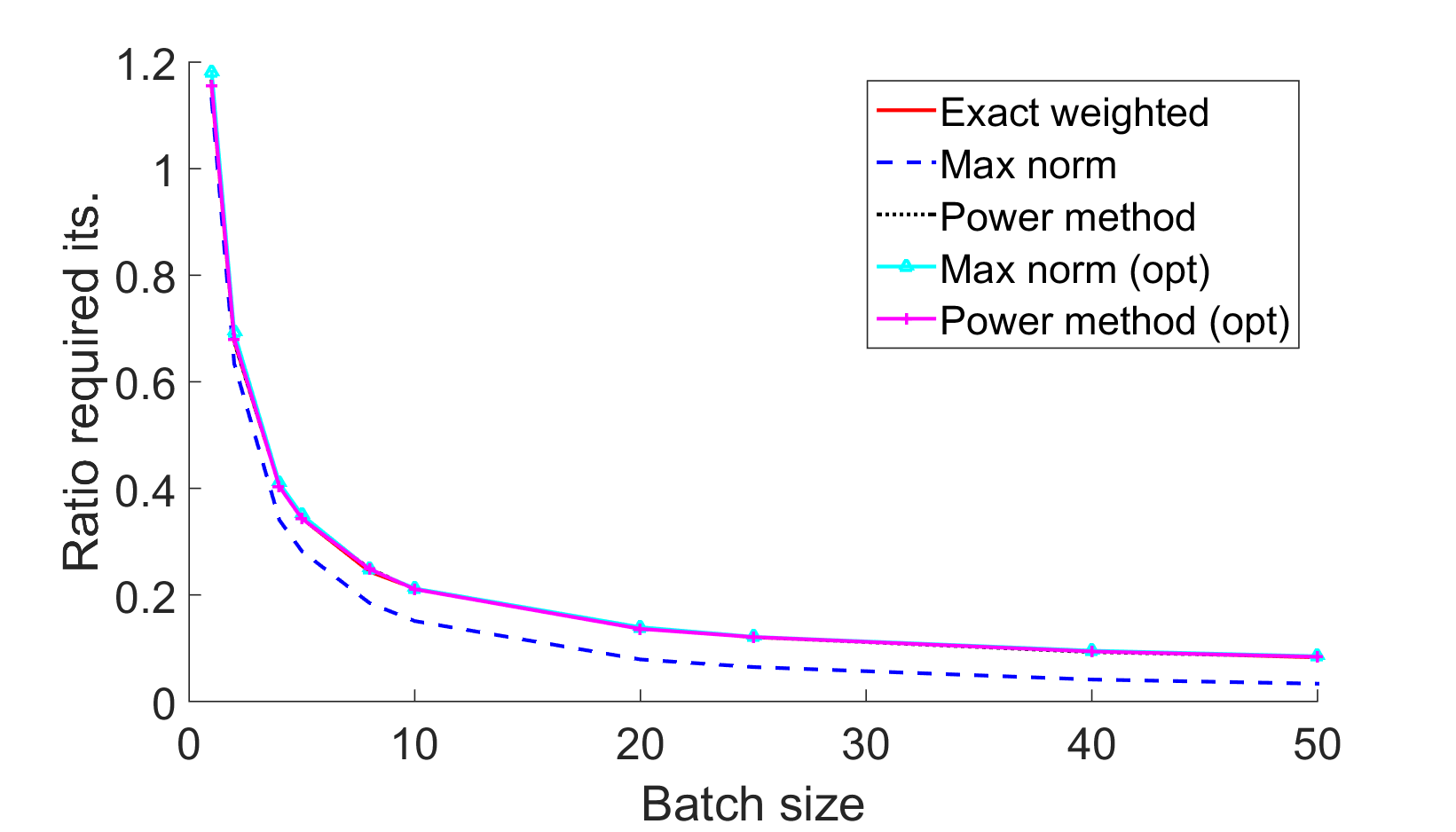}
\caption{\textbf{(Gaussian linear systems: convergence)} Mini-batch SGD on a Gaussian $1000\times 50$ system with various batch sizes; batches created randomly at onset.  Graphs show mean L2-error versus iterations (over 40 trials). Step size $\gamma$ used on each batch was as given in \eqref{LSstep} for the weighted cases and as in \cite[Corollary 3.2]{needell2014stochastic} for the uniform comparisons, where in all cases corresponding approximations were used to compute the spectral norms.  Top: Batches are selected using proposed weighted selection strategy \eqref{w:batchweights}.  Bottom: Ratio of the number of iterations required to reach an error of $10^{-5}$ for weighted batched SGD versus classical (single functional) uniform (unweighted) SGD.  The notation ``(opt)'' signifies that the optimal step size was used, rather than the approximation.  }\label{fig1}
\end{figure}

\begin{figure}[!ht]
\includegraphics[width=2.1in]{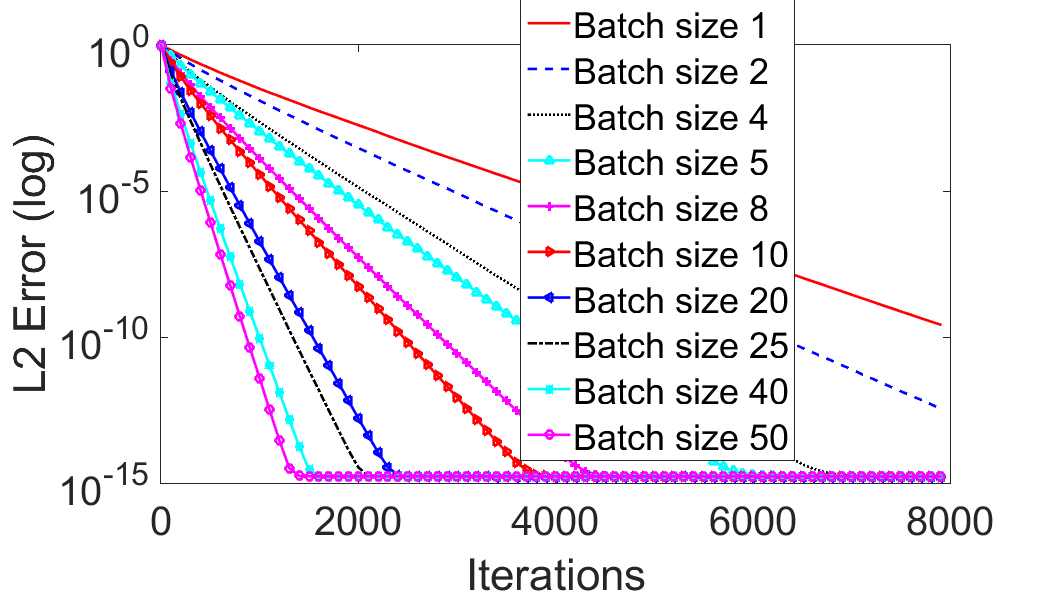} \hspace{-0.1in}\includegraphics[width=2.1in]{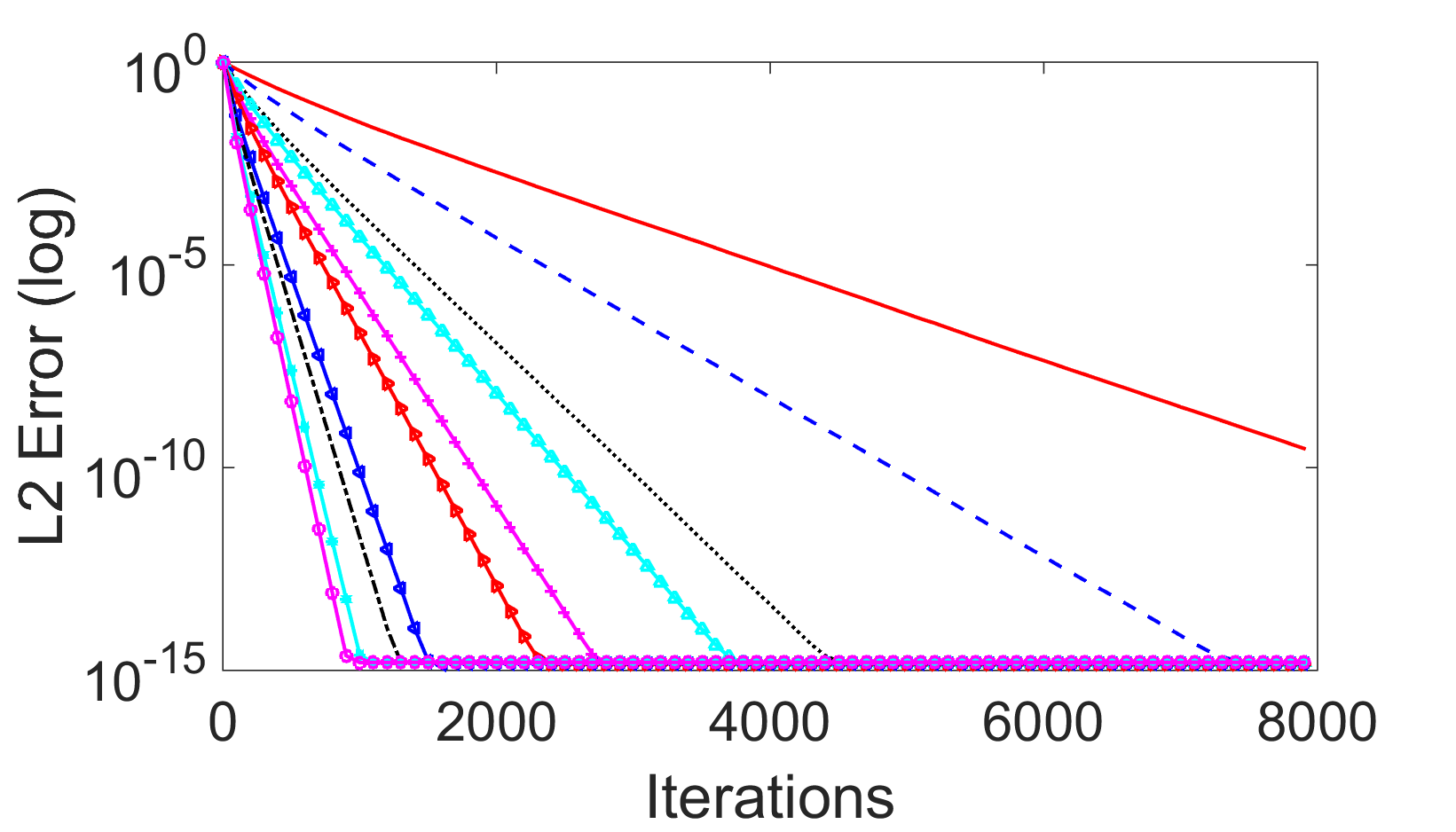}\hspace{-0.1in}\includegraphics[width=2.1in]{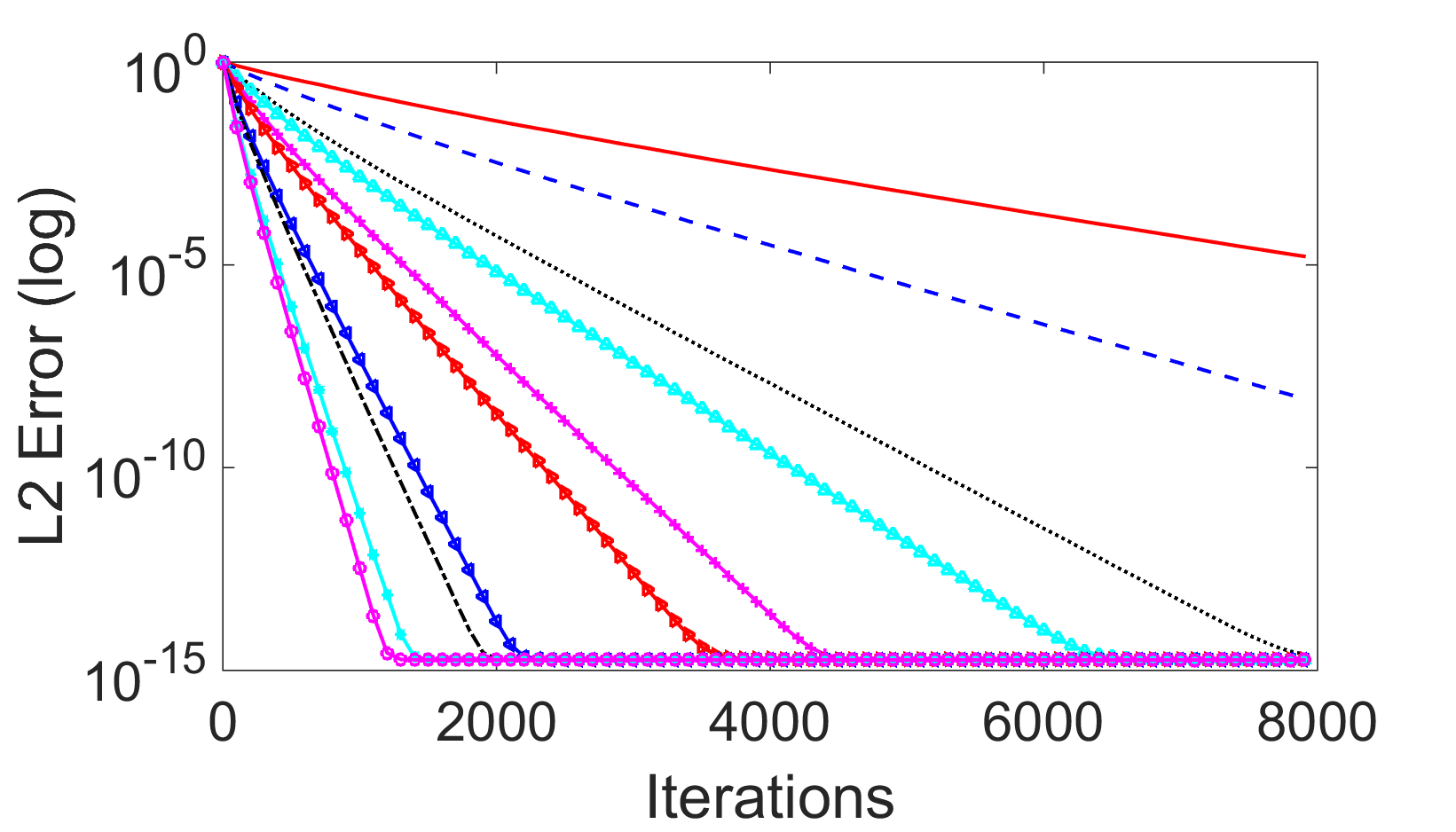}
\caption{\textbf{(Gaussian linear systems with variation: convergence)} Mini-batch SGD on a Gaussian $1000\times 50$ system whose entries in row $k$ have variance $k^2$, with various batch sizes.  Graphs show mean L2-error versus iterations (over 40 trials). Step size $\gamma$ used on each batch was as given in \eqref{LSstep} for weighted SGD and the optimal step size as in \cite[Corollary 3.2]{needell2014stochastic} for uniform sampling SGD.  Top Left: Batches are created randomly at onset, then selected using weighted sampling.  Top Right: Batches are created sequentially at onset, then selected using weighted sampling. Bottom: Batches are created sequentially at onset, then selected using uniform (unweighted) sampling.}\label{fig4}
\end{figure}

\begin{figure}[!ht]
\includegraphics[width=3.2in]{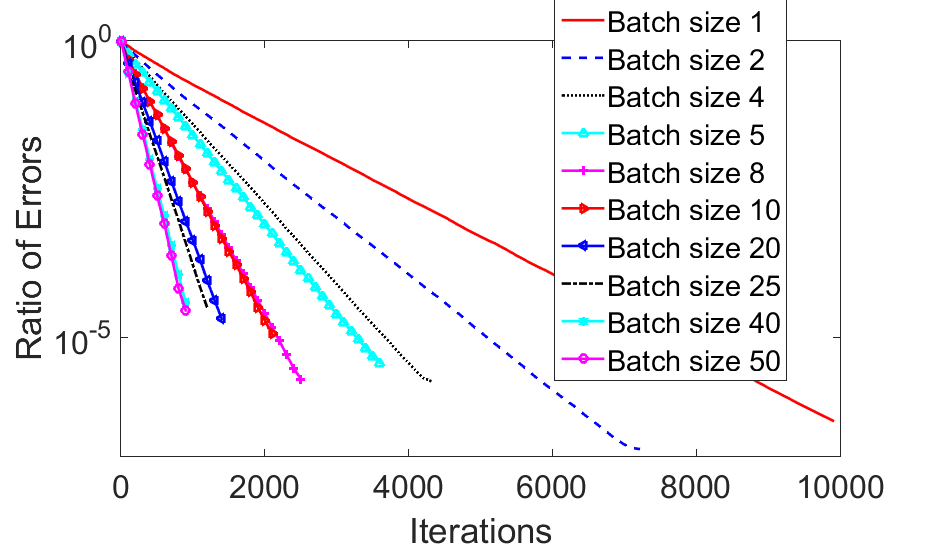} \hspace{-0.1in}\includegraphics[width=3.2in]{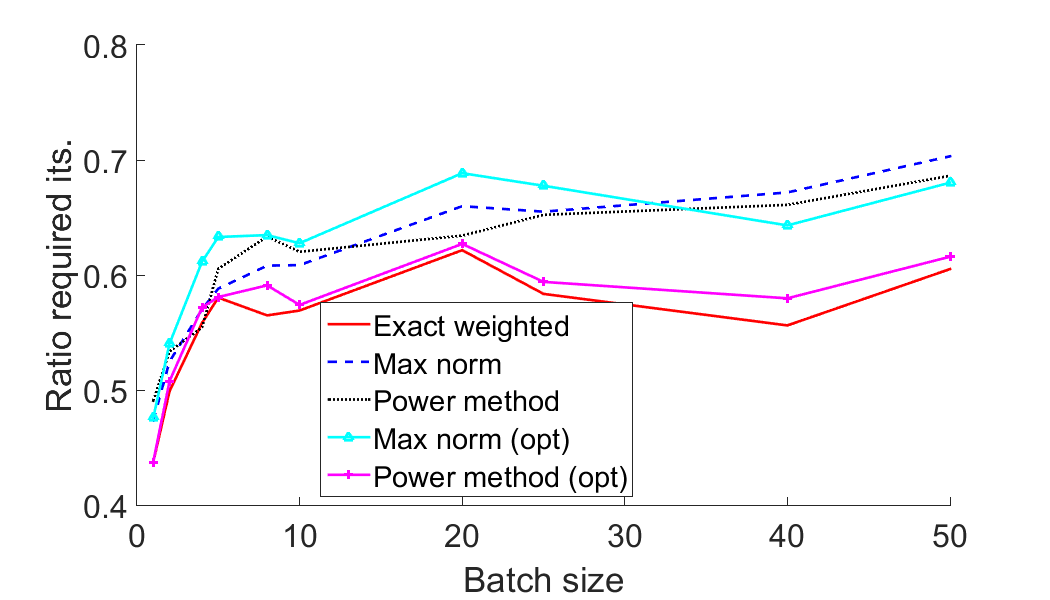}
\caption{\textbf{(Gaussian linear systems with variation: effect of weighting)} Mini-batch SGD on a Gaussian $1000\times 50$ system  whose entries in row $k$ have variance $k^2$, with various batch sizes; batches created sequentially at onset.  Step size $\gamma$ used on each batch was set as in \eqref{LSstep} (approximations for Lipschitz constants also apply to the step size computation) for the weighted cases and as in \cite[Corollary 3.2]{needell2014stochastic} for the uniform cases.  Top: Ratio of mean L2-error using weighted versus unweighted random batch selection (improvements appear when plot is less than one).  Bottom: Ratio of the number of iterations required to reach an error of $10^{-5}$ for various weighted selections versus unweighted random selection.  The notation ``(opt)'' signifies that the optimal step size was used, rather than the approximation. }\label{fig5}
\end{figure}

\begin{figure}[!ht]
\includegraphics[width=3.2in]{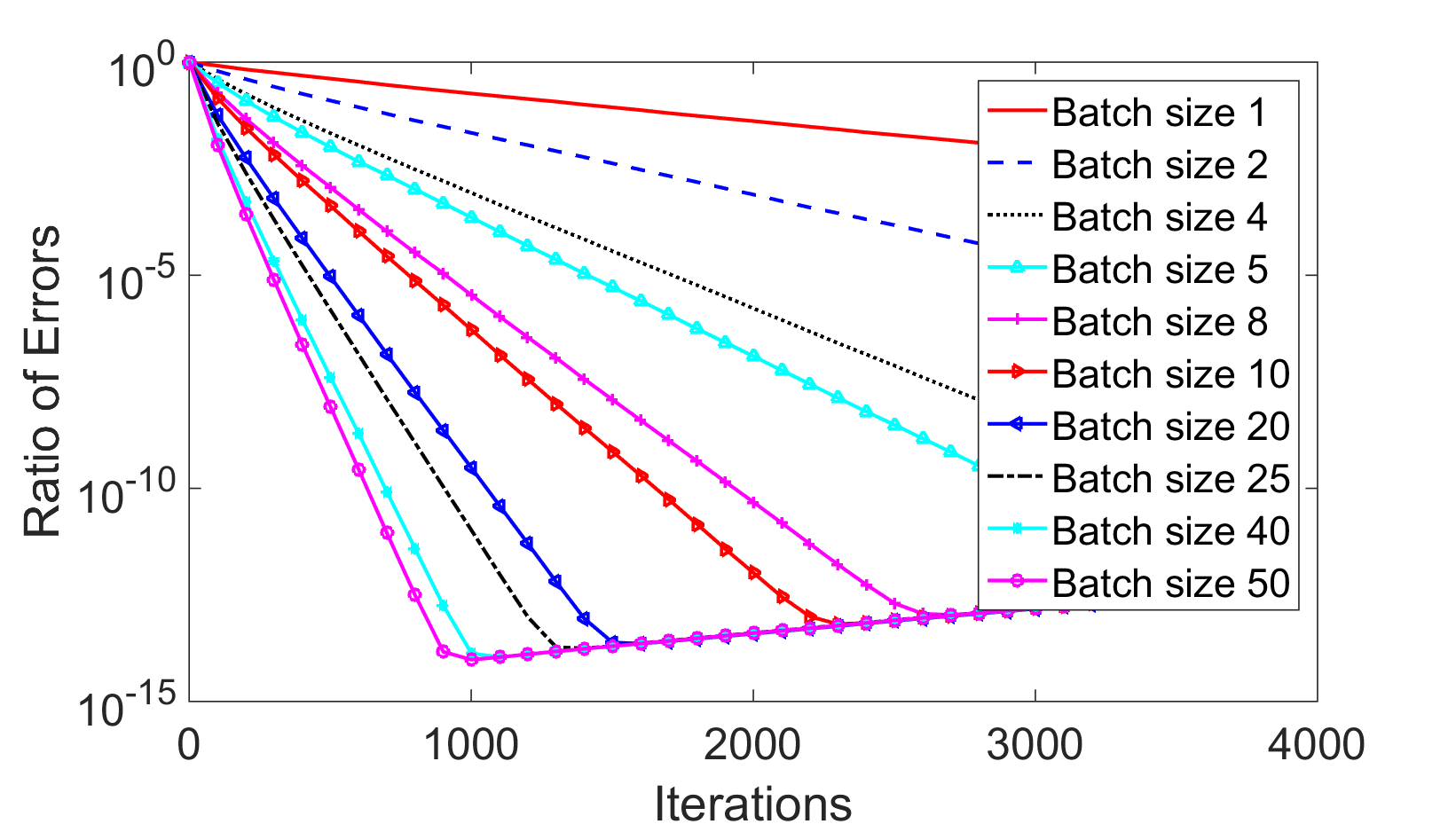} \hspace{-0.1in}\includegraphics[width=3.2in]{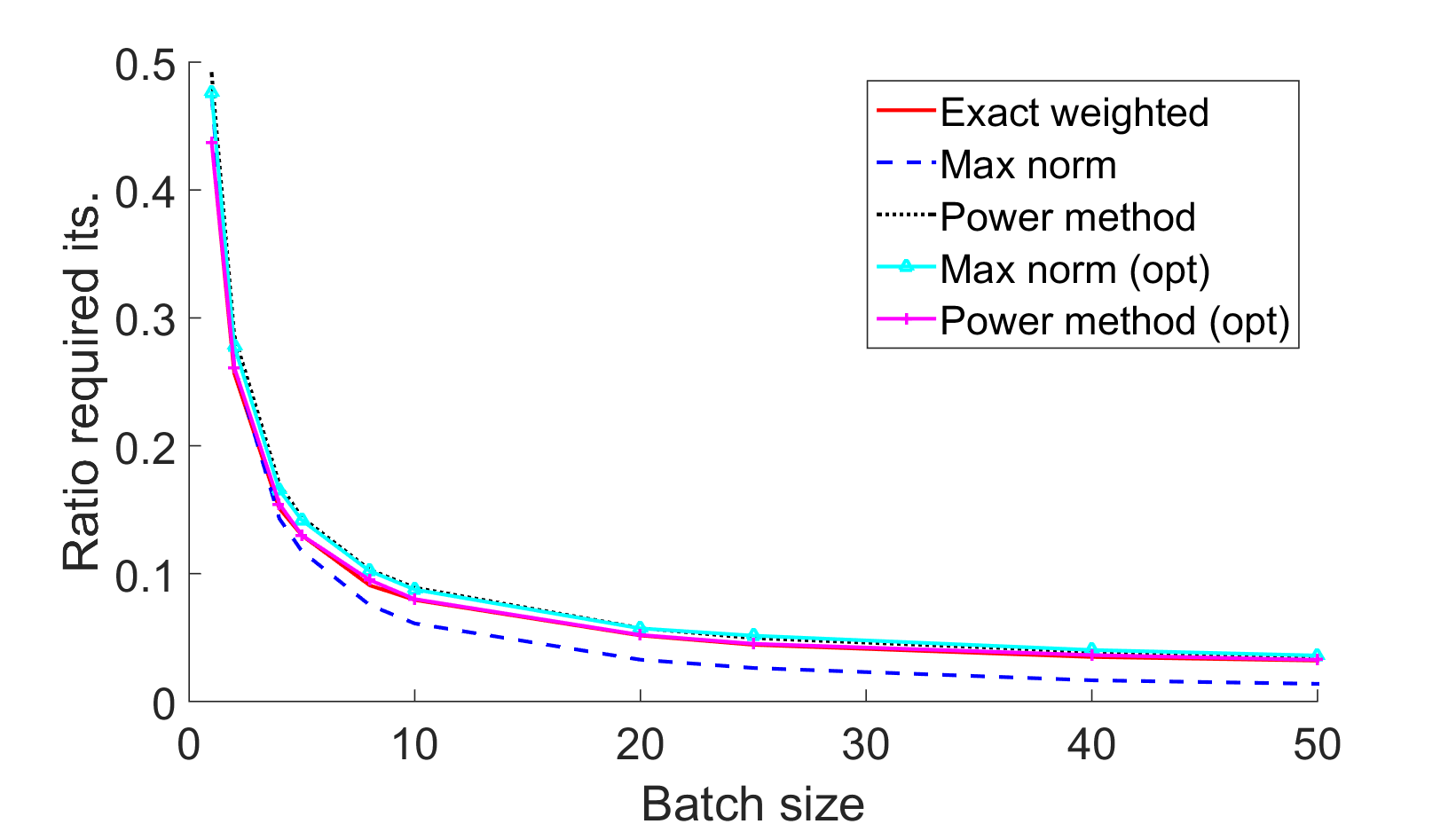}
\caption{\textbf{(Gaussian linear systems with variation: effect of batching)} Mini-batch SGD on a Gaussian $1000\times 50$ system  whose entries in row $k$ have variance $k^2$, with various batch sizes; batches created sequentially at onset.  Step size $\gamma$ used on each batch was set as in \eqref{LSstep} (approximations for Lipschitz constants also apply to the step size computation) for the weighted cases and as in \cite[Corollary 3.2]{needell2014stochastic} for the uniform cases.  Top: Ratio of mean L2-error using weighted batched SGD versus classical (single functional) weighted SGD (improvements appear when plot is less than one).  Bottom: Ratio of the number of iterations required to reach an error of $10^{-5}$ for various weighted selections with batched SGD versus classical (single functional) uniform (unweighted) SGD.  The notation ``(opt)'' signifies that the optimal step size was used, rather than the approximation. }\label{fig6}
\end{figure}

\begin{figure}[!ht]
\includegraphics[width=2.1in]{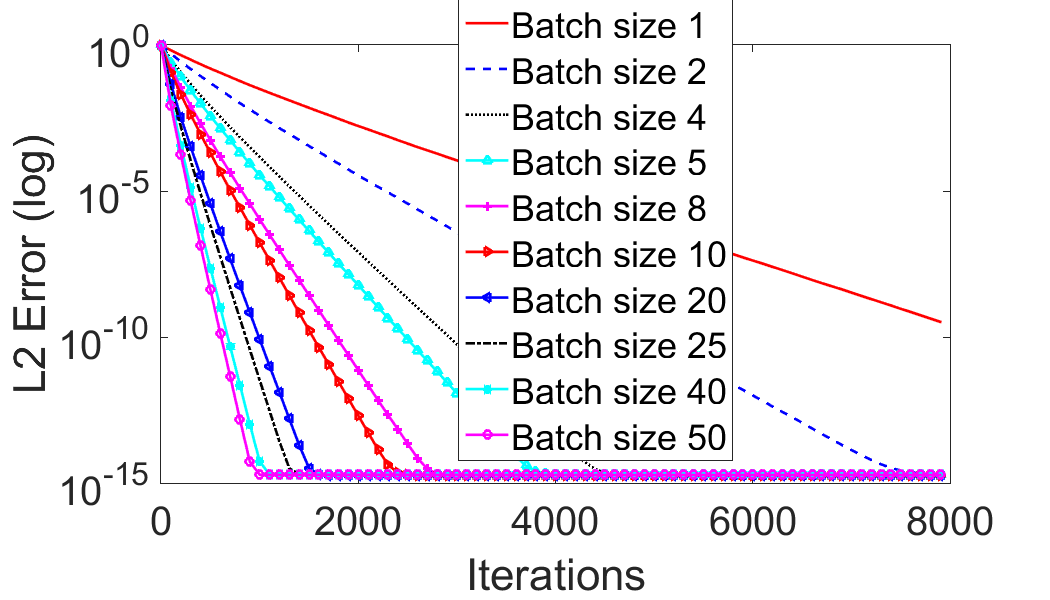} \hspace{-0.1in}\includegraphics[width=2.1in]{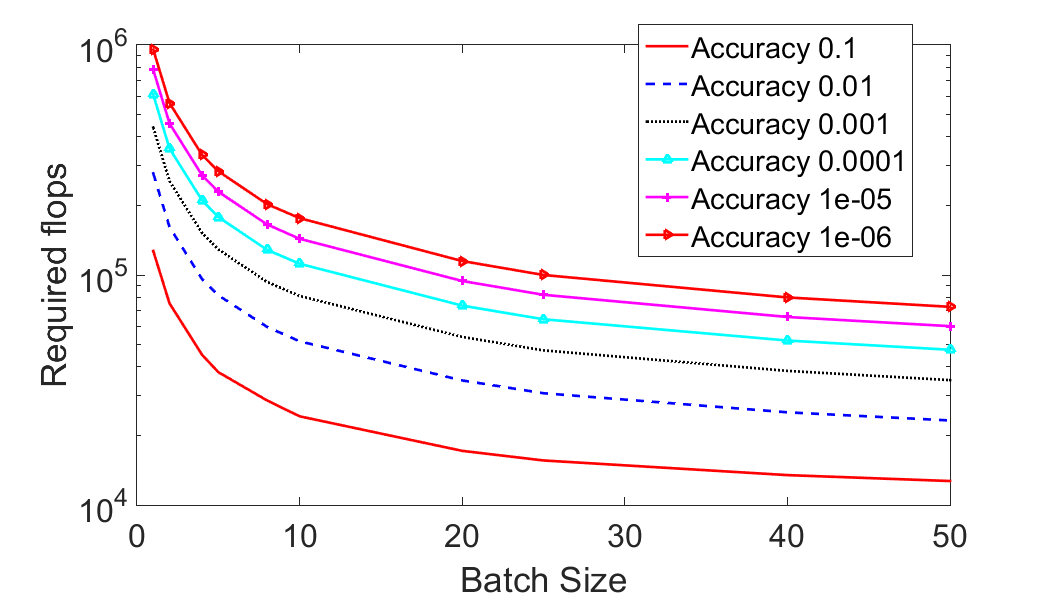}\hspace{-0.1in}\includegraphics[width=2.1in]{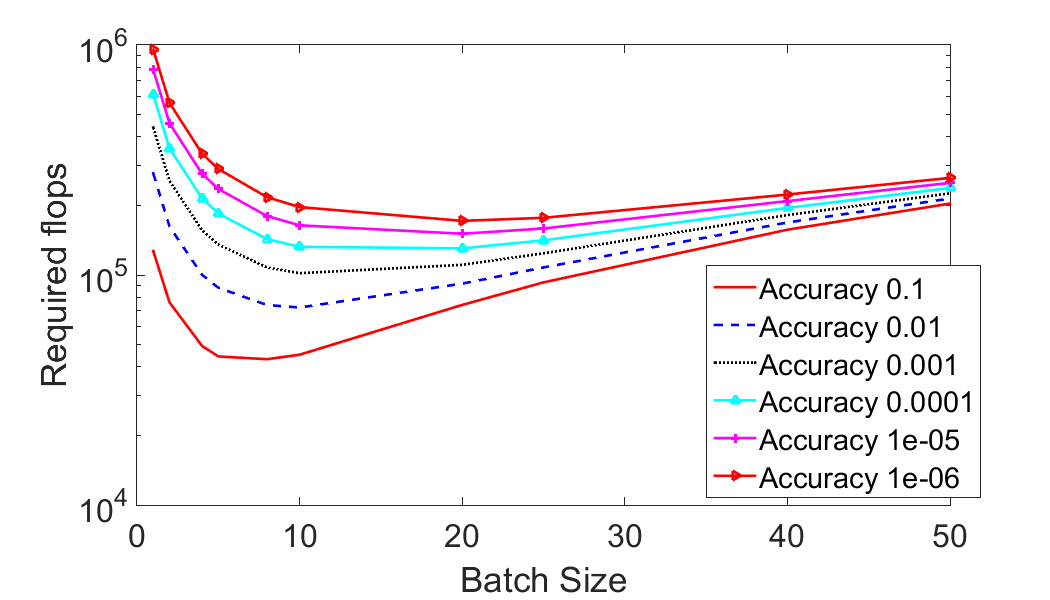}
\caption{\textbf{(Gaussian linear systems with variation: using power method)} Mini-batch SGD on a Gaussian $1000\times 50$ system  whose entries in row $k$ have variance $k^2$, with various batch sizes; batches created sequentially at onset.  Step size $\gamma$ used on each batch was set as in \eqref{LSstep} (approximations for Lipschitz constants also apply to the step size computation) for the weighted cases and as in \cite[Corollary 3.2]{needell2014stochastic} for the uniform cases.  Lipschitz constants for batches are approximated by using $\epsilon^{-1}\log(\epsilon^{-1}b)$ (with $\epsilon=0.01$) iterations of the power method. Top Left: Convergence of the batched method.  Next: Required number of computational flops to achieve a specified accuracy as a function of batch size when computation is shared over $b$ cores (top right) or done on a single node (bottom). }\label{fig6b}
\end{figure}

\begin{figure}[!ht]
\includegraphics[width=2.1in]{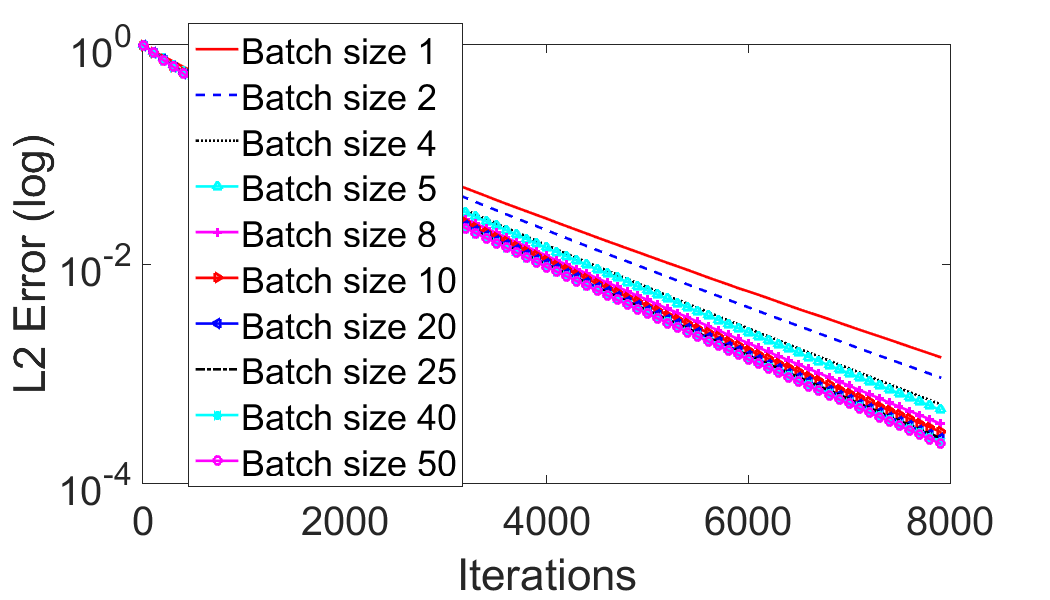} \hspace{-0.1in}\includegraphics[width=2.1in]{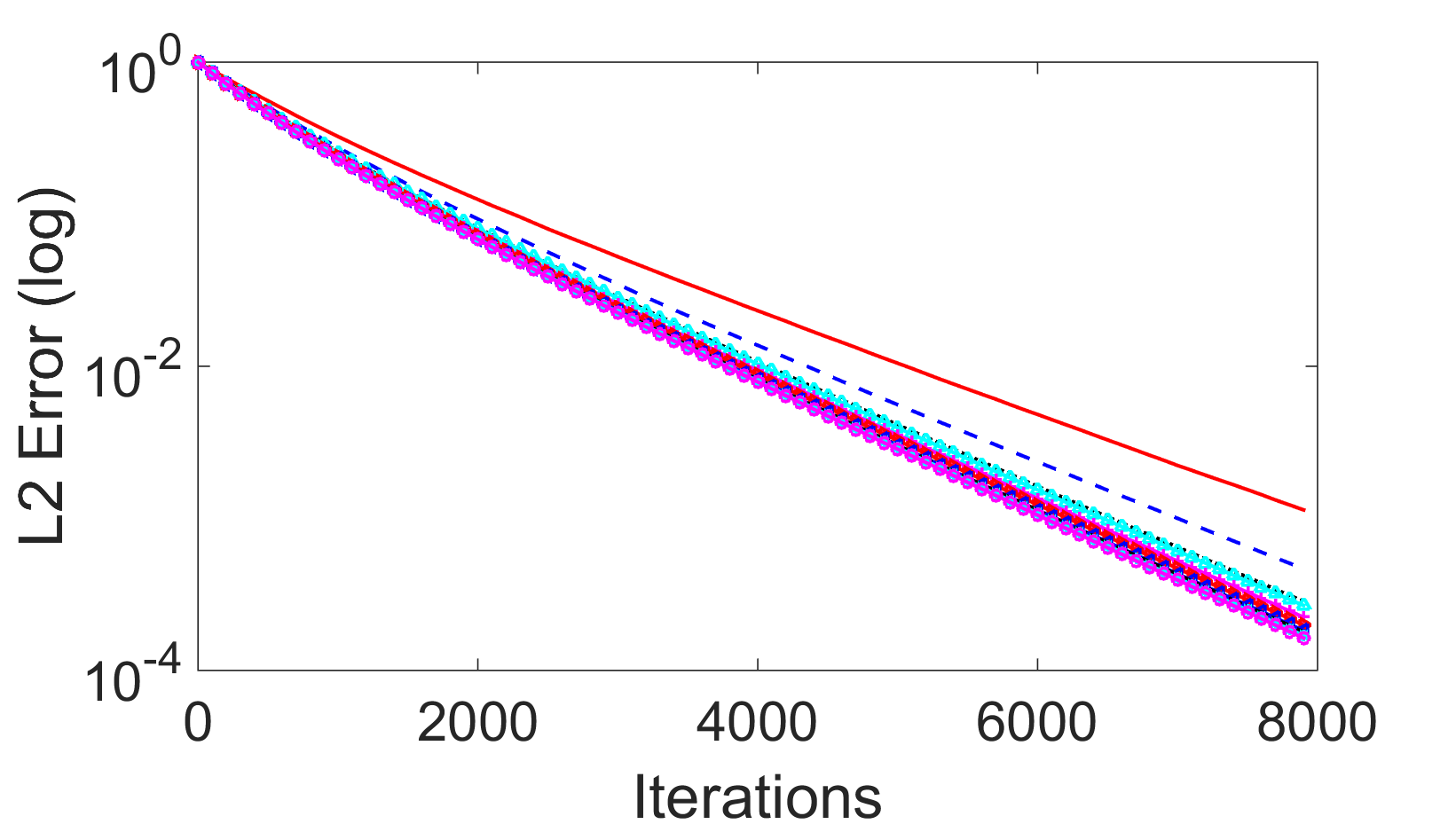}\hspace{-0.1in}\includegraphics[width=2.1in]{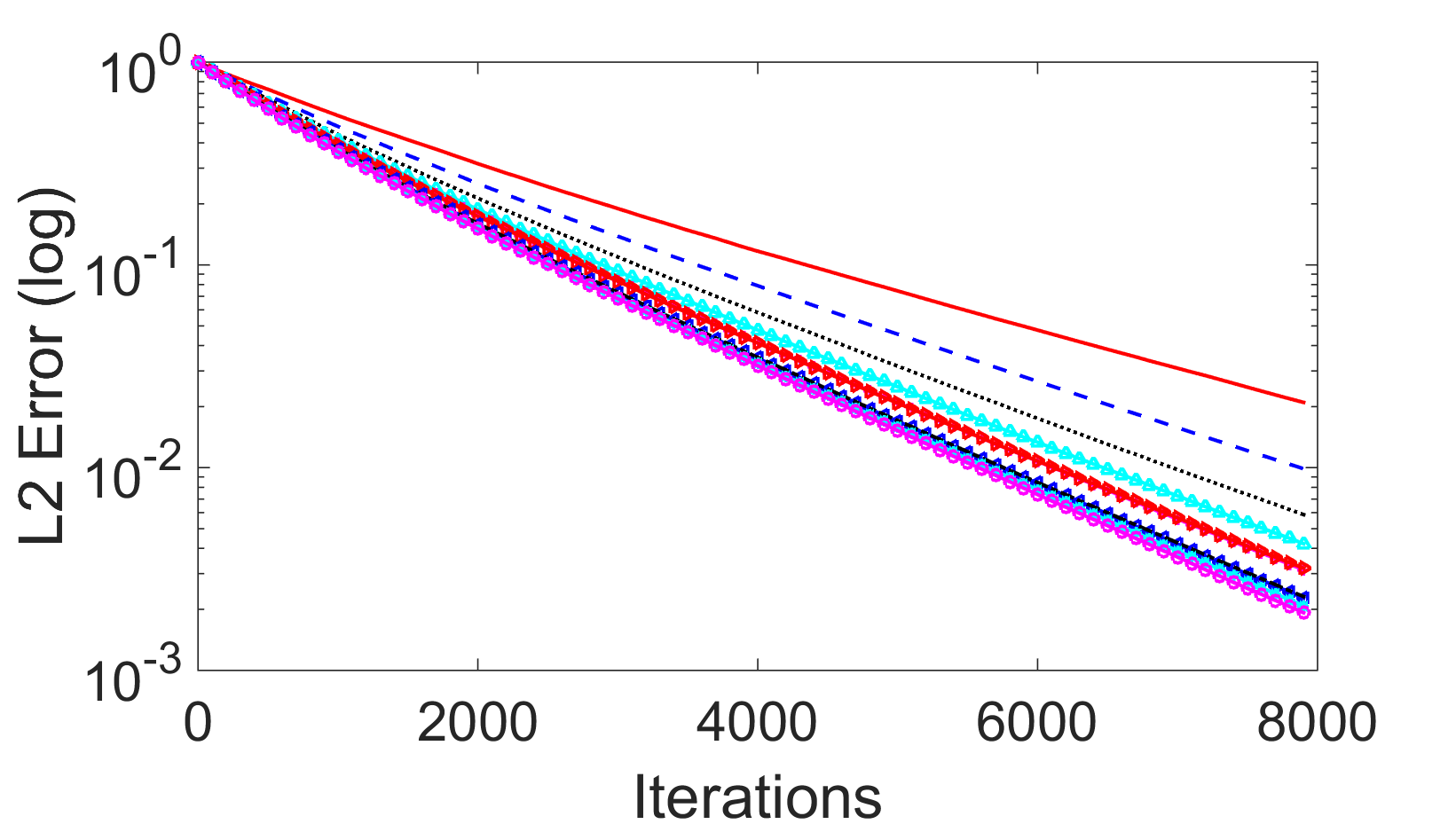}
\caption{\textbf{(Correlated systems with variation: convergence)} Mini-batch SGD on a uniform $1000\times 50$ system whose entries in row $k$ have variance $k^2$, with various batch sizes.  Graphs show mean L2-error versus iterations (over 40 trials). Step size $\gamma$ used on each batch was set as in \eqref{LSstep} (approximations for Lipschitz constants also apply to the step size computation) for the weighted cases and as in \cite[Corollary 3.2]{needell2014stochastic} for the uniform cases.  Top Left: Batches are created randomly at onset, then selected using weighted sampling.  Top Right: Batches are created sequentially at onset, then selected using weighted sampling. Bottom: Batches are created sequentially at onset, then selected using uniform (unweighted) sampling.}\label{figcor}
\end{figure}

\begin{figure}[!ht]
\includegraphics[width=2.1in]{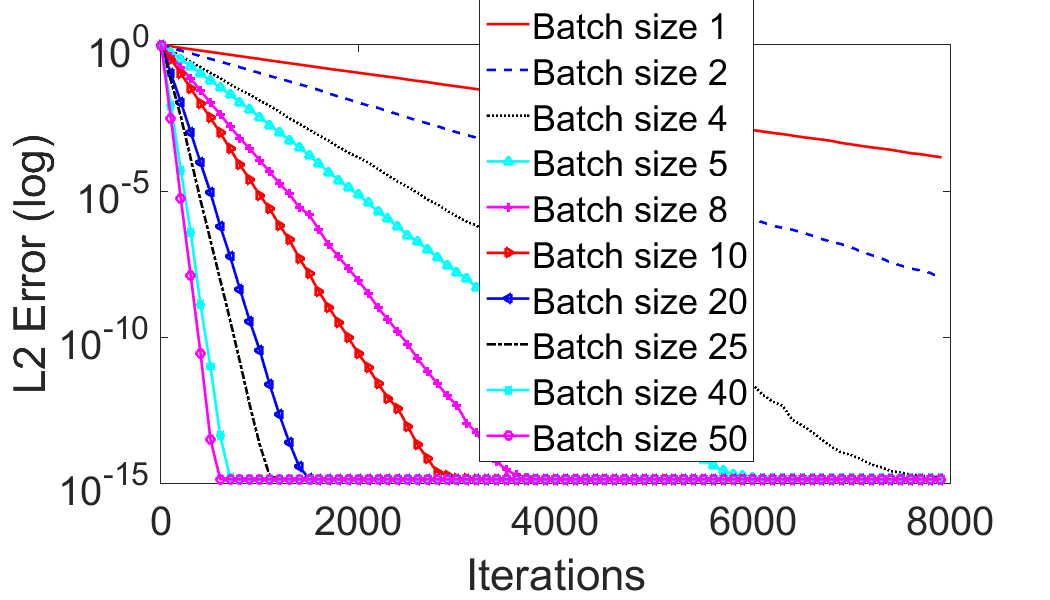} \hspace{-0.1in}\includegraphics[width=2.1in]{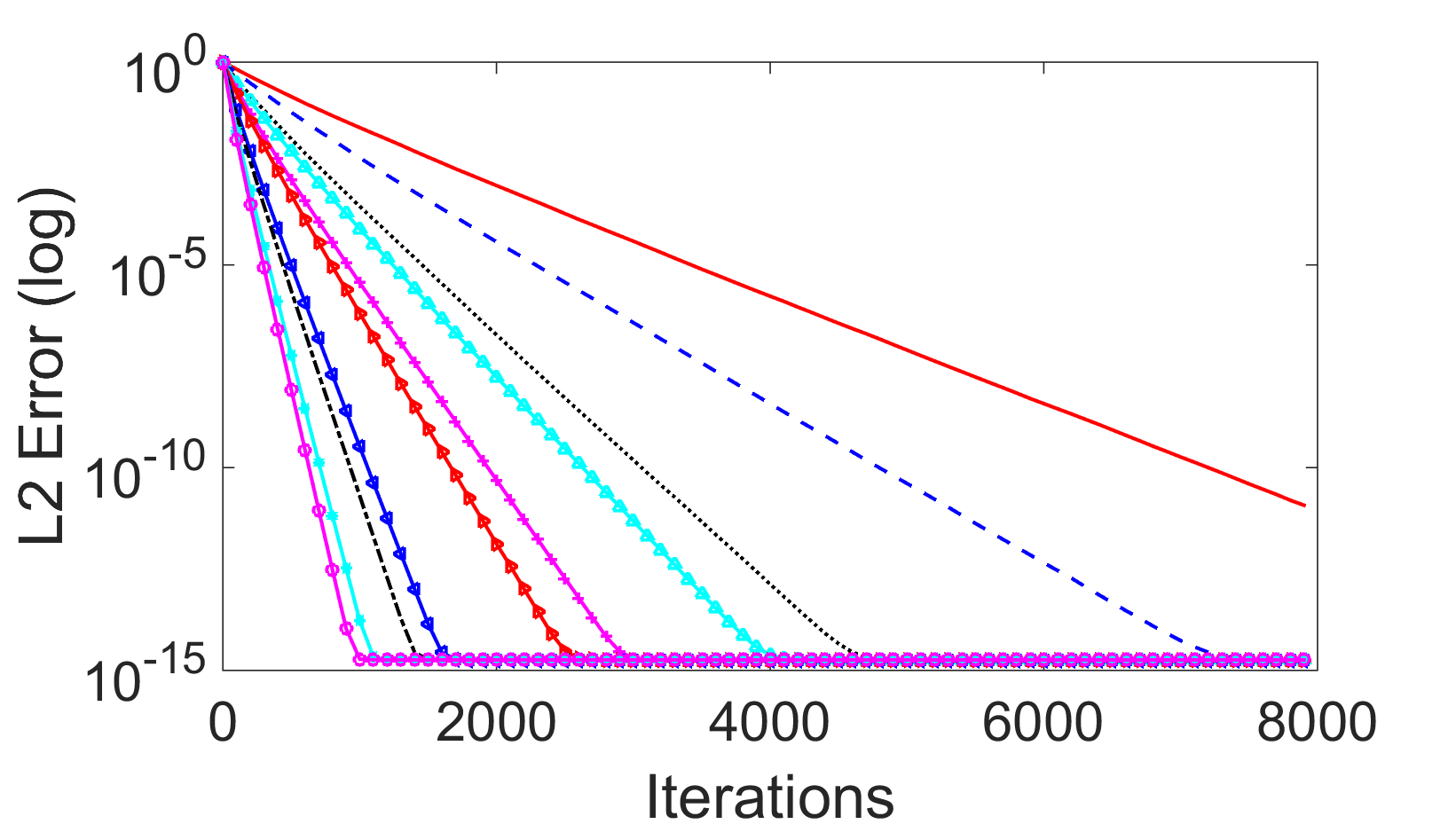}\hspace{-0.1in}\includegraphics[width=2.1in]{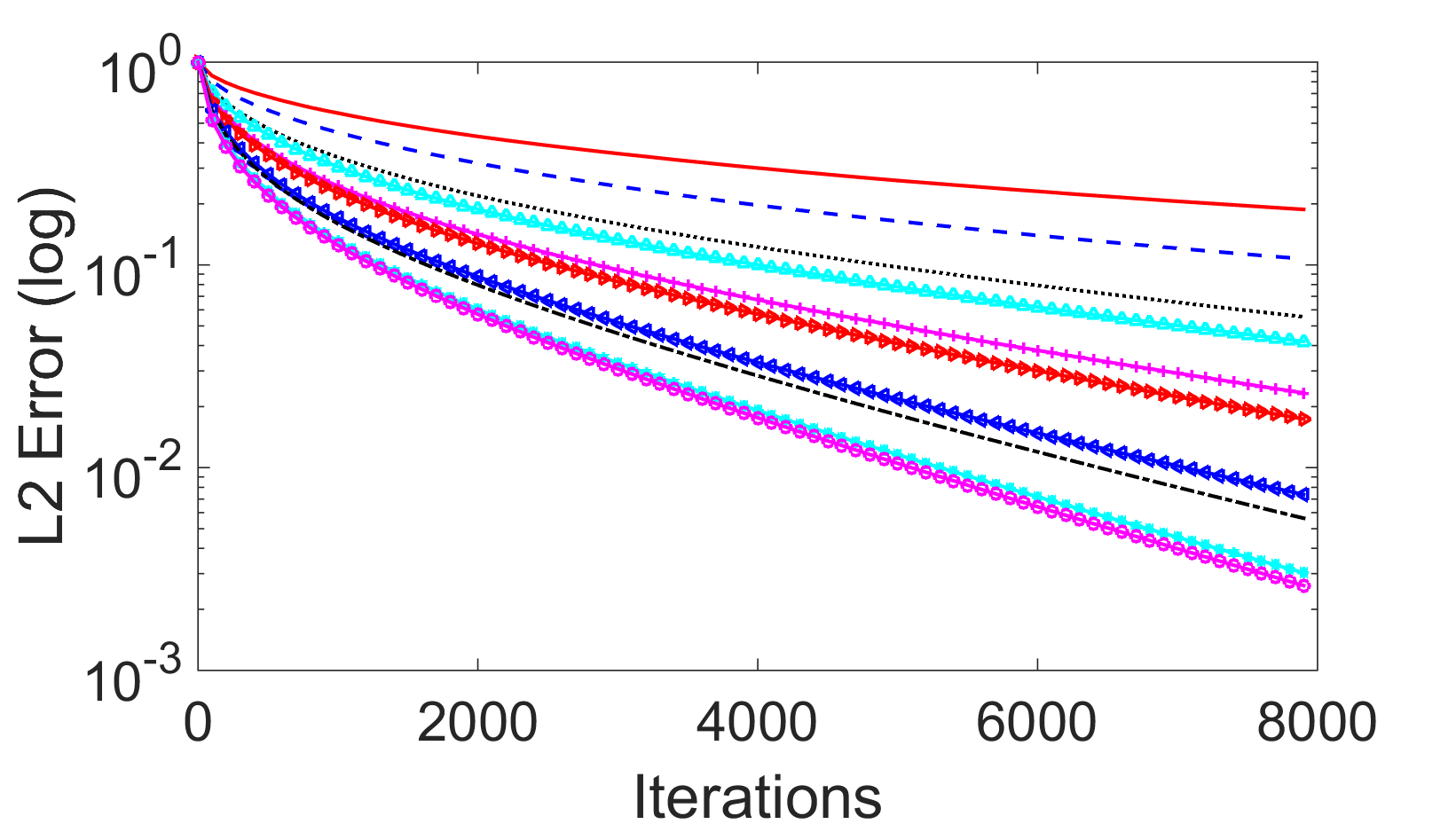}
\caption{\textbf{(Orthonormal, sparse, and tomography systems: convergence)} Mini-batch SGD on two systems for various batch sizes; batches created randomly at onset.  Graphs show mean L2-error versus iterations (over 40 trials). Step size $\gamma$ used on each batch was set as in \eqref{LSstep}.  Top Left: Matrix is a $200\times 200$ orthonormal discrete Fourier transform (DFT).  Top Right: $1000\times 50$ matrix is a sparse standard normal matrix with density $20\%$.  Bottom: Tomography data ($1200\times 400$ system). }\label{fig8}
\end{figure}

\begin{figure}[!ht]
\includegraphics[width=3.2in]{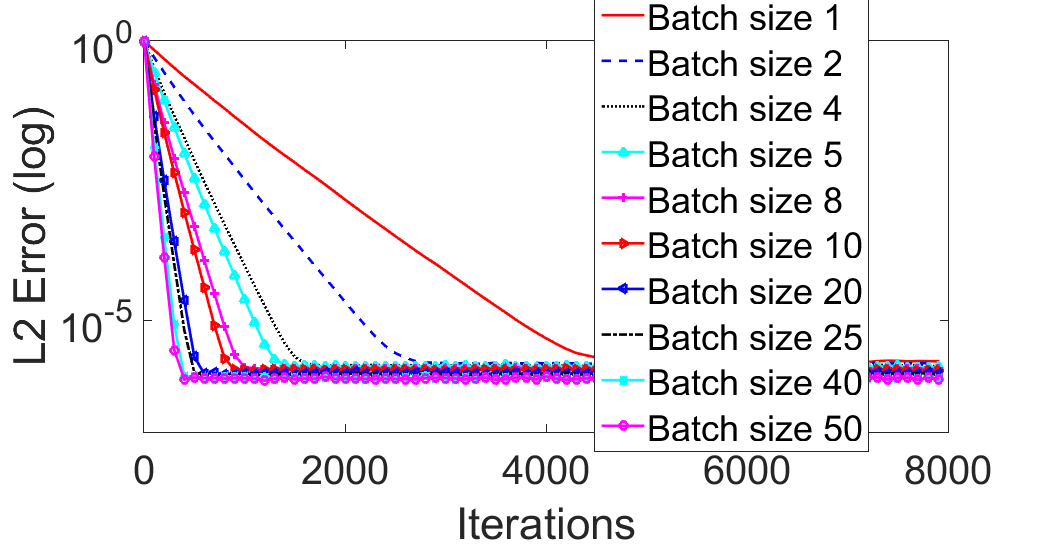} \hspace{-0.1in}\includegraphics[width=3.2in]{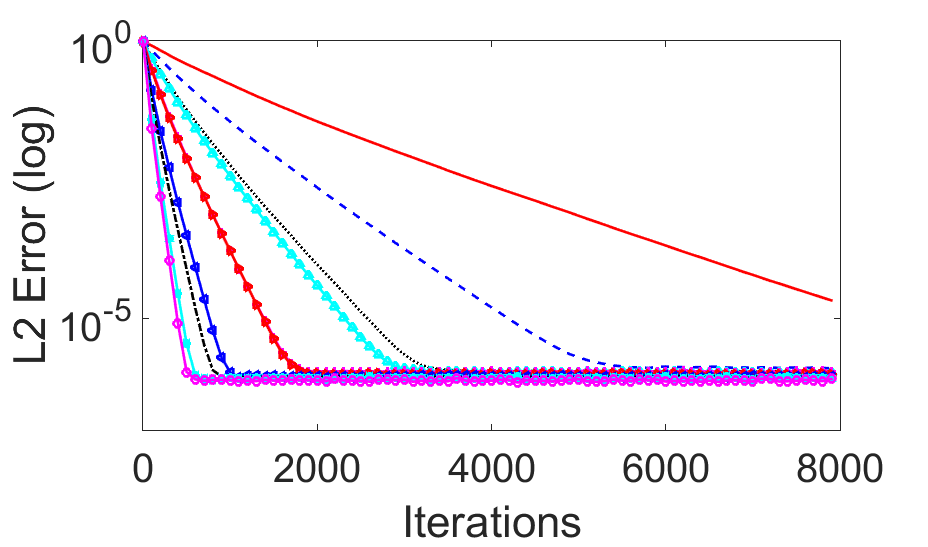}\\
\includegraphics[width=3.2in]{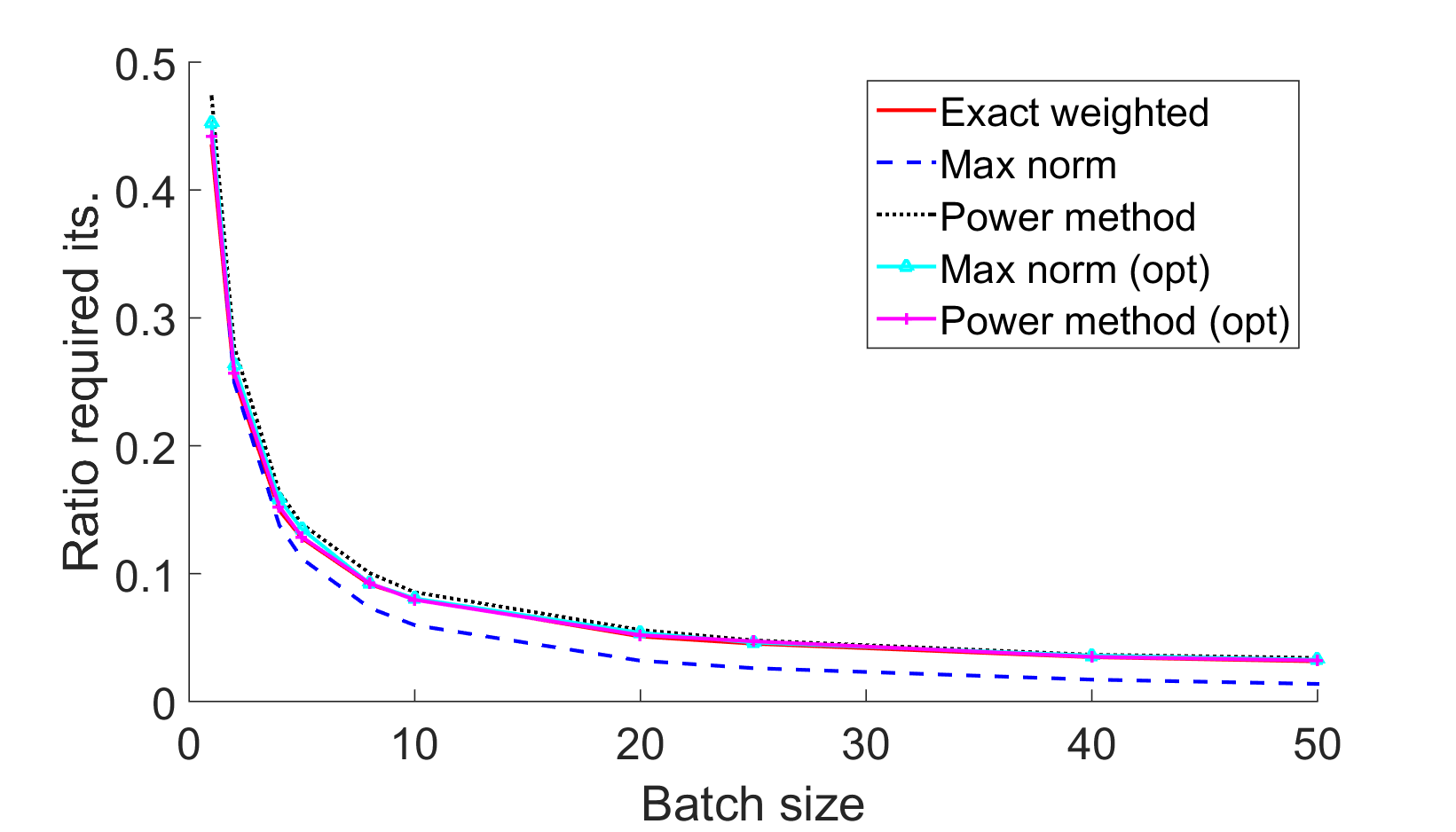}\hspace{-0.1in}\includegraphics[width=3.2in]{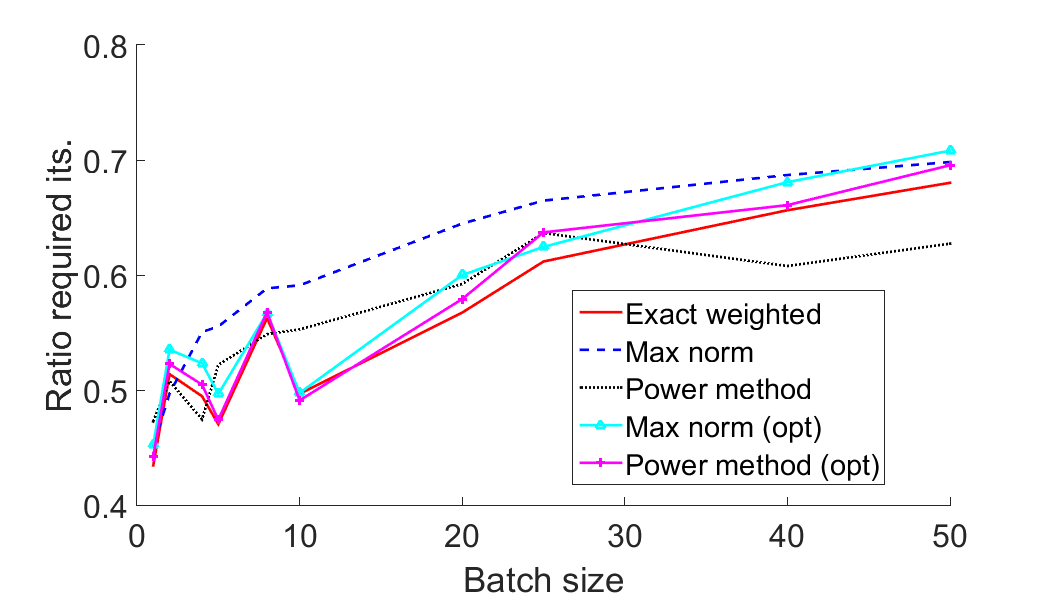}\\
\caption{\textbf{(Noisy systems: convergence)} Mini-batch SGD on a Gaussian $1000\times 50$ system whose entries in row $k$ have variance $k^2$, with various batch sizes.  Noise of norm $1$ is added to system to create an inconsistent system.  Graphs show mean L2-error versus iterations (over 40 trials). Step size $\gamma$ used on each batch was set as in \eqref{LSstep} for the weighted case and as in \cite[Corollary 3.2]{needell2014stochastic} for the uniform case; the residual $\bA\xls - \bb$ was upper bounded by a factor of 1.1 in all cases.    Top: Batches are created sequentially at onset, then selected using weighted sampling. Second Plot: Batches are created sequentially at onset, then selected using uniform (unweighted) sampling.  Third Plot: Ratio of the number of iterations required to reach an error of $10^{-5}$ for various weighted selections with batched SGD versus classical (single functional) uniform (unweighted) SGD.  Bottom: Ratio of the number of iterations required to reach an error of $10^{-5}$ for various weighted selections with batched SGD versus classical uniform (unweighted) SGD as a function of batch size. }\label{fignoise}
\end{figure}

\section{Conclusion}\label{sec:conc}
We have demonstrated that using a weighted sampling distribution along with batches of functionals in SGD can be viewed as complementary approaches to accelerating convergence.  We analyzed the benefits of this combined framework for both smooth and non-smooth functionals, and outlined the specific convergence guarantees for the smooth least squares problem and the non-smooth hinge loss objective.  We discussed several computationally efficient approaches to approximating the weights needed in the proposed sampling distributions and showed that one can still obtain approximately the same improved convergence rate.  We confirmed our theoretical arguments with experimental evidence that highlight in many important settings one can obtain significant acceleration, especially when batches can be computed in parallel. In this parallel setting, we of course see that the improvement increases as the batch size increases, meaning that one should unsurprisingly take advantage of all the cores available. However, we also notice that there may be a tradeoff in computation when the weighting scheme needs to be calculated a priori, and that a non-trivial optimal batch size may exist in that case.  It will be interesting future work to optimize the batch size and other parameters when the parallel computing must be done asynchronously, or in other types of geometric architectures.

%\section*{Acknowledgements}

%
\section*{Acknowledgements}
The authors would like to thank Anna Ma for helpful discussions about this paper, and the reviewers for their thoughtful feedback. Needell was partially supported by NSF CAREER grant $\#1348721$ and the Alfred P. Sloan Foundation.  Ward was partially supported by NSF CAREER grant $\#1255631$. 

%
%\section*{Appendix}
%\addcontentsline{toc}{section}{Appendix}
%%
%%
%When placed at the end of a chapter or contribution (as opposed to at the end of the book), the numbering of tables, figures, and equations in the appendix section continues on from that in the main text. Hence please \textit{do not} use the \verb|appendix| command when writing an appendix at the end of your chapter or contribution. If there is only one the appendix is designated ``Appendix'', or ``Appendix 1'', or ``Appendix 2'', etc. if there is more than one.
%
%\begin{equation}
%a \times b = c
%\end{equation}
%
%\input{referenc}

\bibliography{../rk}

\begin{thebibliography}{10}

\bibitem{needell2014stochastic}
D.~Needell, N.~Srebro, and R.~Ward,
\newblock ``Stochastic gradient descent and the randomized kaczmarz
  algorithm,''
\newblock {\em Mathematical Programming Series A}, vol. 155, no. 1, pp.
  549--573, 2016.

\bibitem{zhao2014stochastic}
P.~Zhao and T.~Zhang,
\newblock ``Stochastic optimization with importance sampling for regularized
  loss minimization,''
\newblock in {\em Proceedings of the 32nd International Conference on Machine
  Learning (ICML-15)}, 2015.

\bibitem{cotter2011better}
A.~Cotter, O.~Shamir, N.~Srebro, and K.~Sridharan,
\newblock ``Better mini-batch algorithms via accelerated gradient methods,''
\newblock in {\em Advances in neural information processing systems}, 2011, pp.
  1647--1655.

\bibitem{agarwal2011distributed}
A.~Agarwal and J.~C. Duchi,
\newblock ``Distributed delayed stochastic optimization,''
\newblock in {\em Advances in Neural Information Processing Systems}, 2011, pp.
  873--881.

\bibitem{dekel2012optimal}
O.~Dekel, R.~Gilad-Bachrach, O.~Shamir, and L.~Xiao,
\newblock ``Optimal distributed online prediction using mini-batches,''
\newblock {\em The Journal of Machine Learning Research}, vol. 13, no. 1, pp.
  165--202, 2012.

\bibitem{takavc2013mini}
M.~Takac, A.~Bijral, P.~Richtarik, and N.~Srebro,
\newblock ``Mini-batch primal and dual methods for {S}{V}{M}s,''
\newblock in {\em Proceedings of the 30th International Conference on Machine
  Learning (ICML-13)}, 2013, vol.~3, pp. 1022--1030.

\bibitem{robmon}
H.~Robbins and S.~Monroe,
\newblock ``A stochastic approximation method,''
\newblock {\em Ann. Math. Statist.}, vol. 22, pp. 400--407, 1951.

\bibitem{bottou}
L.~Bottou and O.~Bousquet,
\newblock ``The tradeoffs of large-scale learning,''
\newblock {\em Optimization for Machine Learning}, p. 351, 2011.

\bibitem{bottou2010large}
L.~Bottou,
\newblock ``Large-scale machine learning with stochastic gradient descent,''
\newblock in {\em Proceedings of COMPSTAT'2010}, pp. 177--186. Springer, 2010.

\bibitem{njls09}
A.~Nemirovski, A.~Juditsky, G.~Lan, and A.~Shapiro,
\newblock ``Robust stochastic approximation approach to stochastic
  programming,''
\newblock {\em SIAM Journal on Optimization}, vol. 19, no. 4, pp. 1574--1609,
  2009.

\bibitem{shalev2008svm}
S.~Shalev-Shwartz and N.~Srebro,
\newblock ``{S}{V}{M} optimization: inverse dependence on training set size,''
\newblock in {\em Proceedings of the 25th international conference on Machine
  learning}, 2008, pp. 928--935.

\bibitem{SV09:Randomized-Kaczmarz}
T.~Strohmer and R.~Vershynin,
\newblock ``A randomized {K}aczmarz algorithm with exponential convergence,''
\newblock {\em J. Fourier Anal. Appl.}, vol. 15, no. 2, pp. 262--278, 2009.

\bibitem{Nee10:Randomized-Kaczmarz}
D.~Needell,
\newblock ``Randomized {K}aczmarz solver for noisy linear systems,''
\newblock {\em BIT}, vol. 50, no. 2, pp. 395--403, 2010.

\bibitem{bach2011}
F.~Bach and E.~Moulines,
\newblock ``Non-asymptotic analysis of stochastic approximation algorithms for
  machine learning,''
\newblock {\em Advances in Neural Information Processing Systems (NIPS)}, 2011.

\bibitem{nesterov2012efficiency}
Y.~Nesterov,
\newblock ``Efficiency of coordinate descent methods on huge-scale optimization
  problems,''
\newblock {\em SIAM J. Optimiz.}, vol. 22, no. 2, pp. 341--362, 2012.

\bibitem{richtak14}
P.~Richt{\'a}rik and M.~Tak{\'a}{\v{c}},
\newblock ``On optimal probabilities in stochastic coordinate descent
  methods,''
\newblock {\em Optimization Letters}, pp. 1--11, 2015.

\bibitem{qrz15}
Z.~Qu, P.~Richtarik, and T.~Zhang,
\newblock ``Quartz: Randomized dual coordinate ascent with arbitrary
  sampling,''
\newblock in {\em Advances in neural information processing systems}, 2015,
  vol.~28, pp. 865--873.

\bibitem{cqr15}
D.~Csiba, Z.~Qu, and P.~Richtarik,
\newblock ``Stochastic dual coordinate ascent with adaptive probabilities,''
\newblock {\em Proceedings of the 32nd International Conference on Machine
  Learning (ICML-15)}, 2015.

\bibitem{lee2013efficient}
Y.~T. Lee and A.~Sidford,
\newblock ``Efficient accelerated coordinate descent methods and faster
  algorithms for solving linear systems,''
\newblock in {\em Foundations of Computer Science (FOCS), 2013 IEEE 54th Annual
  Symposium on}. IEEE, 2013, pp. 147--156.

\bibitem{schmidt2013minimizing}
M.~Schmidt, N.~Roux, and F.~Bach,
\newblock ``Minimizing finite sums with the stochastic average gradient,''
\newblock {\em arXiv preprint arXiv:1309.2388}, 2013.

\bibitem{xiao2014proximal}
L.~Xiao and T.~Zhang,
\newblock ``A proximal stochastic gradient method with progressive variance
  reduction,''
\newblock {\em SIAM Journal on Optimization}, vol. 24, no. 4, pp. 2057--2075,
  2014.

\bibitem{defossez2015averaged}
A.~D{\'e}fossez and F.~R. Bach,
\newblock ``Averaged least-mean-squares: Bias-variance trade-offs and optimal
  sampling distributions.,''
\newblock in {\em AISTATS}, 2015.

\bibitem{shalev2011pegasos}
S.~Shalev-Shwartz, Y.~Singer, N.~Srebro, and A.~Cotter,
\newblock ``Pegasos: Primal estimated sub-gradient solver for svm,''
\newblock {\em Mathematical programming}, vol. 127, no. 1, pp. 3--30, 2011.

\bibitem{byrd2012sample}
R.~H. Byrd, G.~M. Chin, J.~Nocedal, and Y.~Wu,
\newblock ``Sample size selection in optimization methods for machine
  learning,''
\newblock {\em Mathematical programming}, vol. 134, no. 1, pp. 127--155, 2012.

\bibitem{NW12:Two-Subspace-Projection}
D.~Needell and R.~Ward,
\newblock ``Two-subspace projection method for coherent overdetermined linear
  systems,''
\newblock {\em Journal of Fourier Analysis and Applications}, vol. 19, no. 2,
  pp. 256--269, 2013.

\bibitem{konevcny2014ms2gd}
J.~Konecn{\`y}, J.~Liu, P.~Richtarik, and M.~Takac,
\newblock ``ms2gd: Mini-batch semi-stochastic gradient descent in the proximal
  setting,''
\newblock {\em IEEE Journal of Selected Topics in Signal Processing}, vol. 10,
  no. 2, pp. 242--255, 2016.

\bibitem{li2014efficient}
M.~Li, T.~Zhang, Y.~Chen, and A.~J. Smola,
\newblock ``Efficient mini-batch training for stochastic optimization,''
\newblock in {\em Proceedings of the 20th ACM SIGKDD international conference
  on Knowledge discovery and data mining}. ACM, 2014, pp. 661--670.

\bibitem{csiba2016importance}
D.~Csiba and P.~Richtarik,
\newblock ``Importance sampling for minibatches,''
\newblock {\em arXiv preprint arXiv:1602.02283}, 2016.

\bibitem{gower2016randomized}
R.~M. Gower and P.~Richt{\'a}rik,
\newblock ``Randomized quasi-newton updates are linearly convergent matrix
  inversion algorithms,''
\newblock {\em arXiv preprint arXiv:1602.01768}, 2016.

\bibitem{RefWorks:48}
E.~J. Cand\`es and T.~Tao,
\newblock ``Decoding by linear programming,''
\newblock {\em IEEE T. Inform. Theory}, vol. 51, pp. 4203--4215, 2005.

\bibitem{klein1996efficient}
P.~Klein and H.-I. Lu,
\newblock ``Efficient approximation algorithms for semidefinite programs
  arising from max cut and coloring,''
\newblock in {\em Proceedings of the twenty-eighth annual ACM symposium on
  Theory of computing}. ACM, 1996, pp. 338--347.

\bibitem{nes04}
Y.~Nesterov,
\newblock {\em Introductory Lectures on Convex Optimization},
\newblock Kluwer, 2004.

\bibitem{shamir2012stochastic}
O.~Shamir and T.~Zhang,
\newblock ``Stochastic gradient descent for non-smooth optimization:
  Convergence results and optimal averaging schemes,''
\newblock {\em arXiv preprint arXiv:1212.1824}, 2012.

\bibitem{rakhlin2011making}
A.~Rakhlin, O.~Shamir, and K.~Sridharan,
\newblock ``Making gradient descent optimal for strongly convex stochastic
  optimization,''
\newblock {\em arXiv preprint arXiv:1109.5647}, 2012.

\bibitem{yang2016weighted}
J.~Yang, Y.-L. Chow, C.~R{\'e}, and M.~W. Mahoney,
\newblock ``Weighted sgd for $\ell_p$ regression with randomized
  preconditioning,''
\newblock in {\em Proceedings of the Twenty-Seventh Annual ACM-SIAM Symposium
  on Discrete Algorithms}. SIAM, 2016, pp. 558--569.

\bibitem{hansen2007regularization}
P.~C. Hansen,
\newblock ``Regularization tools version 4.0 for matlab 7.3,''
\newblock {\em Numer. Algorithms}, vol. 46, no. 2, pp. 189--194, 2007.

\end{thebibliography}
\bibliographystyle{IEEEbib_deanna}

\end{document}